\documentclass{amsart}

\usepackage{amsmath}
\usepackage{amsfonts}
\usepackage{amssymb}
\usepackage{amsthm}
\usepackage[mathscr]{eucal}
\usepackage{graphicx}
\usepackage{xcolor}
\usepackage{datetime}
\usepackage{enumerate}

\usepackage{hyperref}
\hypersetup{colorlinks,
	filecolor=black,
	linkcolor=blue,
	citecolor=darkgray,
	urlcolor=green,
	bookmarksopen=true}
\usepackage{url}
\usepackage{datetime}
\usepackage{enumerate}
\usepackage[utf8]{inputenc}
\usepackage{lineno}

\newcommand{\mb}{\mathbb}
\newcommand{\mc}{\mathcal}
\newcommand{\mr}{\mathrm}
\newcommand{\cv}{\nabla}

\newcommand{\ve}{\varepsilon}
\newcommand{\vp}{\varphi}

\newcommand{\lp}{\langle}
\newcommand{\rp}{\rangle}

\newcommand{\ten}{\otimes}
\newcommand{\KN}{\mathbin{\bigcirc\mspace{-15mu}\wedge\mspace{3mu}}}

\newcommand{\mat}{\begin{pmatrix}}
\newcommand{\rix}{\end{pmatrix}}

\def \blue {\color{blue}}

\DeclareMathOperator{\Rc}{Rc}
\DeclareMathOperator{\Rm}{Rm}

\newtheorem{theorem}{Theorem}

\newtheorem{lemma}[theorem]{Lemma}
\newtheorem{corollary}[theorem]{Corollary}

\newtheorem{remark}[theorem]{Remark}

\newtheorem{assumpA}{Assumption}

\begin{document}


\title[Compact multiply warped Ricci flow solutions]{Local singularities of compact multiply warped Ricci flow solutions}

\author{James Isenberg}
\address[James Isenberg]{University of Oregon}
\email{isenberg@uoregon.edu}
\urladdr{http://www.uoregon.edu/$\sim$isenberg/}

\author{Dan Knopf}
\address[Dan Knopf]{University of Texas at Austin}
\email{danknopf@math.utexas.edu}
\urladdr{http://www.ma.utexas.edu/users/danknopf}

\author{Zilu Ma}
\address[Zilu Ma]{Rutgers University}
\email{zm227@math.rutgers.edu}

\author{Nata\v sa \v Se\v sum}
\address[Nata\v sa \v Se\v sum]{Rutgers University}
\email{natasas@math.rutgers.edu}
\urladdr{http://www.math.rutgers.edu/$\sim$natasas/}

\begin{abstract}
We demonstrate that any four-dimensional shrinking Ricci soliton $(\mc B \times {\mb S^2}, g)$, where $\mc B$ is any two-dimensional complete noncompact surface 
and $g$ is a warped product metric over the base $\mc B$, has to be isometric to the generalized cylinder $\mb R^2\times\mb S^2$ equipped with the standard 
cylindrical metric. After completing this classification, we study Ricci flow solutions that are multiply warped products --- but not products --- and provide rigorous 
examples of the formation of generalized cylinder singularity models $\mb R^k\times\mb S^\ell$.
\end{abstract}

\thanks{DK thanks the Simons Foundation for support from Award 635293. N\v S thanks the NSF for support in DMS 2105508.}

\maketitle

\tableofcontents

\section{Introduction}
Our goals in this paper are to {\bf (i)} classify four-dimensional shrinking Ricci solitons $(\mc B^2 \times \mb S^2, g)$, where $\mc B^2$ is a general two-dimensional complete noncompact surface, and $g$ is a warped product metric over $\mc B^2$, and then {\bf (ii)} study the formation of Ricci flow singularity models that are generalized cylinders of the form $\mb R^k\times\mb S^\ell$, with $k\geq1$
and $\ell\geq2$. While these models are expected to occur and arise trivially from product solutions, we provide rigorous analyses of singularities that form from
solutions that are not products. Our motivations are as follows.

Ricci solitons, manifolds $(\mc M,g)$ with $\Rc[g]+\lambda g +\frac12 \mc L_X g=0$, are generalized stationary solutions that frequently arise as singularity
models of the flow, that is, as limits of parabolic dilations of finite-time singularities. Indeed, for Type-I singularities, which are conjecturally generic for compact
solutions and satisfy $\sup(T-t)|\Rm|<\infty$, one can by~\cite{EMT11} always extract a nontrivial gradient shrinking soliton limit; \emph{i.e.,} one with $\lambda<0$
and $X=\mr{grad}(f)$ for a nonconstant potential function $f$.

For mean curvature flow (MCF), a series of influential papers by Colding and Minicozzi starting with~\cite{CM12} prove that the only stable shrinking MCF
solitons are generalized cylinders $\mb R^k\times\mb S^\ell$. For Ricci flow, somewhat less is known, but the reader should consult the work in the recent~\cite{CM22}.
In any case, generalized Ricci flow cylinders $\mb R^k\times\mb S^\ell$, with $k\geq1$ and $\ell\geq2$, form an important subclass of shrinking soliton
singularity models.

Of these Ricci flow limits, almost all rigorous non-product examples studied to date are neckpinches for which $k=1$. Our intent in this work is to 
provide complementary analyses of singularity formations that yield limits with $k\geq2$. To do so, we employ a multiply warped product \emph{Ansatz}
used previously in~\cite{CIKS22}.

\subsection{Summary of our main results}

In Section \ref{sec-Uhlenbeck}, we classify warped product four-dimensional Ricci flow shrinkers on $\mc B^2\times \mb S^2$. 

 \begin{theorem}
\label{thm-classification}
Let $(\mathcal{B}^2\times\mathbb{S}^2, g)$ be a noncompact and nonflat shrinking Ricci soliton, where $g = \check g + v^2 g_{\mathbb{S}^2}$ and
$v : \mathcal{B}^2\to (0,\infty)$. Then $(\mathcal{B}^2\times\mathbb{S}^2, g)$ is isometric to the {\bf bubble sheet} (generalized cylinder
$\mathbb{R}^2\times\mathbb{S}^2$ with the standard cylindrical metric).
\end{theorem} 

Next we focus on constructing nontrivial examples of multiply warped product Ricci flows which develop as singularity models in the form of generalized cylinders.
In Section~\ref{Generalized cylinders-1}, we prove the following result for multiply warped product metrics over the base $\mb S^1$.
Essentially, our Main Theorem~\ref{Main-1} demonstrates that nontrivial multiply warped product metrics do, under very mild hypotheses,
develop neckpinch singularities that are qualitatively similar to those analyzed in~\cite{Simon}, \cite{AK04}, \cite{AK07}, \cite{IKS16}, and elsewhere.

\begin{theorem} 	\label{Main-1} Let
$\big(\mc M = \mb S^1\times \mb S^{n_1}_1\times\mb S^{n_2}_2\times\cdots\times\mb S^{n_A}_A,g=\check g+\sum_a v_a^2\,\hat g_{\mb S^{n_a}_a}\big)$
be a Ricci flow solution originating from initial data satisfying
Assumptions~\ref{single fiber pinching}, \ref{guarantee cylinder}, and~\ref{small gradient}, as stated below, with all $n_a\geq2$

Then there is a single fiber, which we may without loss of generality take to be $(\mb S^{n_1}_1,v_1^2\,\hat g_{\mb S^{n_1}_1})$, that becomes singular first.

The solution $g(t)$ develops a Type-I singularity at $T<\infty$ for which the singularity limit is $\big(\mb R^{n_2+\cdots+n_A+1}\times\mb S^{n_1}, g_\infty)$.

The metric $g_\infty$ is a product of the flat Euclidean factor $\mb R^{n_2+\cdots+n_A}$ with the product metric (Gaussian soliton)
on $\mb R^1\times\mb S^{n_1}$ that models the Ricci flow neckpinch.

There exist constants $0<\delta<C<\infty$ such that the
radius $v_1$ of the smallest sphere at distance $\sigma$ from the neckpinch
is bounded from above by
\[
 v_1\leq \sqrt{2(n-1)(T-t)}+\frac{C\sigma^2}{-\log(T-t)\sqrt{T-t}}
\]
for $|\sigma|\leq 2\sqrt{-(T-t)\log(T-t)}$, and by
\[
  v_1\leq  C\frac\sigma{\sqrt{-\log (T-t)}}
           \sqrt{\log\frac\sigma{\sqrt{-(T-t)\log(T-t)}}}
\]
for $2\sqrt{-(T-t)\log(T-t)}\leq\sigma\leq(T-t)^{\frac12-\delta}$.

\end{theorem}
\medskip

In Section~\ref{Generalized cylinders-2}, we prove an analogous result for multiply warped products over any closed two-dimensional base.

\begin{theorem} 	\label{Main-2} Let
$\big(\mc M = \mc B^2\times \mb S^{n_1}_1\times\mb S^{n_2}_2\times\cdots\times\mb S^{n_A}_A,g=\check g+\sum_a v_a^2\,\hat g_{\mb S^{n_a}_a}\big)$
be a Ricci flow solution over a compact surface $(\mc B^2,\check g)$ originating from initial data satisfying
Assumptions~\ref{single fiber pinching}, \ref{eta tame}, \ref{guarantee cylinder}, and~\ref{small gradient}, as stated below,  with all $n_a\geq2$.

Then there is a single fiber, which we may without loss of generality take to be $(\mb S^{n_1}_1,v_1^2\,\hat g_{\mb S^{n_1}_1})$, that becomes singular first.
The solution $g(t)$ develops a Type-I singularity at $T<\infty$ for which the singularity limit  is a direct product
\[
\big(\mb R^{n_2+\cdots+n_A}\times \mc K^{n_1 + 2}, g_\infty = g_{\rm{eucl}} + g_{\mc K}).
\]
Here $g_{\rm{eucl}}$ is a flat Euclidean metric on $\mb R^{n_2+\cdots+n_A}$, and $(\mc K, g_{\mc K})$ is a nonflat gradient shrinking soliton on a complete, noncompact manifold $\mc K = \tilde{\mc B}\times S^{n_1}$, where $\tilde{\mc B}$ is a two-dimensional complete noncompact base, and $g_{\mc K}$ is a warped product metric over the base $\tilde{\mc B}$.
\end{theorem}

If $n_1=2$ is the dimension of the crushed fiber, the following result follows immediately from Theorems~\ref{thm-classification} and \ref{Main-2}.

\begin{corollary}
\label{cor-n1-2}
Let $n_1 = 2$ in Theorem \ref{Main-2}, and let
\[\big(\mc M = \mc B^2\times \mb S^2_1\times\mb S^{n_2}_2\times\cdots\times\mb S^{n_A}_A,g=\check g+\sum_a v_a^2\,\hat g_{\mb S^{n_a}_a}\big)
\]
be a Ricci flow solution over a compact surface $(\mc B^2,\check g)$ originating from initial data satisfying
Assumptions~\ref{single fiber pinching}, \ref{eta tame}, \ref{guarantee cylinder}, and~\ref{small gradient}, as stated below,  with all $n_a\geq2$.
Without loss of generality, our assumptions can be made to guarantee the first fiber  $(\mb S^{2}_1,v_1^2\,\hat g_{\mb S^2_1})$ becomes singular first.

The  solution $g(t)$ develops a Type-I singularity at $T<\infty$ for which the singularity limit  is a generalized  cylinder  $\mb S^2 \times \mb R^{2 + n_2 + \dots n_A}$, with a standard cylindrical metric.
\end{corollary}

We expect that similar results hold under suitable hypotheses for solutions over compact bases of higher dimensions, but we do not study those in this work.

\subsection{Outline of the paper}

In Section~\ref{Geometric preliminaries}, we introduce the multiply warped product \emph{Ansatz} that we employ in this paper.
Then making use of details reviewed in Appendix~\ref{Geometric basics}, we derive two equivalent forms of the Ricci flow system for such metrics
and compute the implied evolution of the scalar curvature of the base.

In Section \ref{Classification} we prove Theorem \ref{thm-classification}. This theorem plays an important role in establishing Corollary~\ref{cor-n1-2}.

In Section~\ref{Elementary estimates}, we establish some preliminary  \emph{a priori} $C^k\;(k=0,1,2)$ estimates for the solutions we study along with
sufficient assumptions for each to hold. These estimates are employed elsewhere in the paper to prove our main results.

As noted above, in Section~\ref{Generalized cylinders-1}, we prove our main result for the formation of generalized cylinder limits forming from multiply warped product
solutions over a a one-dimensional base manifold

Then in Section~\ref{Generalized cylinders-2}, as also noted above, we prove our main result for the formation of generalized cylinder limits forming from
multiply warped product solutions over a two-dimensional base manifold. We also prove Corollary~\ref{cor-n1-2} there.

The appendices provide additional technical details. Specifically, in Appendix~\ref{Geometric basics}, we recall a few useful identities for the geometry
of multiply warped products. In Appendix~\ref{Hessian evolution}, we derive the evolution of the covariant Hessian norm $|\cv^2 v_a|^2$.

\section{The metrics we study}		\label{Geometric preliminaries}

Let $(\mc B^n,g_{\mc B})$ be compact.
For each $a\in\{1,\dots,A\}$,  let $(\mc F_a^{n_a},g_{\mc F_a})$ be a  compact Einstein manifold normalized so that $\Rc[g_{\mc F_a}]=\mu_a g_{\mc F_a}$.
Recall that $\mu_a=n_a-1$ for the standard round sphere of radius $1$. To ensure that singularities form in finite time,we assume that all
fibers have nonnegative Einstein constants $\mu_a\geq0$.

Given smooth functions $v_a:\mc B^n\rightarrow\mb R_+$, one can construct a multiply warped product manifold
\[
\Big(\mc M^N:=\mc B^n\times\mc F_1^{n_1}\times\cdots\times\mc F_A^{n_A},\;g:=g_{\mc B} + \sum_{a=1}^A v_a^2\,g_{\mc F_a}\Big).
\]
This is a Riemannian submersion~\cite{ONB66} but is not a product unless all $v_a$ are constant.
Accordingly, we find it convenient to introduce the notations
\[
\check g = g_{\mc B}, \qquad \hat g_a = g_{\mc F_a}, \quad\mbox{ and }\quad g_a = v_a^2\,\hat g_a,
\]
in order to write the metric as
\begin{equation}	\label{MWP metric}
g=\check g+\sum_{a=1}^A g_a = \check g +\sum_{a=1}^A v_a^2\,\hat g_a.
\end{equation}

We review pertinent details of the geometry of the metric~\eqref{MWP metric} and state our index conventions in Appendix~\ref{Geometric basics}.

\subsection{Notational conventions}

Although we study similar metrics in~\cite{CIKS22}, we warn the reader that the notation used here differs slightly from that in the earlier paper.

Our conventions are consistent with~\cite{CIKS22} in that undecorated quantities like $\Delta$ and $|\cdot|$ are computed with respect to the metric $g$ on the whole space.
We use decorations to indicate if a quantity is computed with respect to the metric $\check g$ on the base or with respect to a metric $\hat g_a$ on a fiber.

The main difference is that, in our earlier work, we write $g=g_{\mc B} + \sum_{a=1}^A u_a\,g_{\mc F_a}$ and study the functions $u_a$,
whereas below, we study $v_a=\sqrt u_a$ and $w_a = \frac12\log u_a$.

\subsection{Evolution by Ricci flow}
Under Ricci flow, the multiply warped product structure is preserved, and the base metric $\check g$ and warping functions $v_a$ evolve by a coupled system.
\medskip

In what follows, we find it convenient to study two equivalent forms of this system. The first is independent of a choice of gauge and uses the Laplacian
$\Delta$ of the full metric:
\begin{subequations}															\label{RF full system}
\begin{align}
 \partial_t\,\check g  &= -2 \check{\Rc} + 2\sum_{a=1}^A n_a v_a^{-1}\check\cv^2 v_a,		 \label{RF base}\\
 \partial_t v_a &=\Delta v_a -\frac{\mu_a+|\cv v_a|^2}{v_a},							 \label{RF fibers}
\end{align}
\end{subequations}
where $\check{\Rc}=\Rc[\check g]$.
\medskip

To derive the second form of the system, we fix a gauge. We begin by defining the vector field
\[
X:= \sum_a n_a v_a^{-1} \cv v_a	\qquad\Rightarrow\qquad
(\mc L_X\check g)_{ij}=2\sum_a n_a\Big\{v_a^{-1}(\check\cv^2 v_a)_{ij}-v_a^{-2}\cv_i v_a \cv_j v_a\Big\}.
\]
Then we find that~\eqref{RF base} is equivalent to 
\[
\partial_t \check g = -2\check{\Rc}+2\Big(\sum_a n_a v_a^{-2}\cv v_a \ten \cv v_a\Big) +\mc L_X\check g.
\]
Similarly, defining $w_a:=\log v_a$, one deduces from~\eqref{RF fibers} that
\[
\partial_t w_a = \check\Delta w_a-\mu_ae^{-2w_a}+\mc L_X w_a.
\]
Hence, as long as a solution remains smooth, one can invert the diffeomorphisms generated by $X$ in order to study the simpler system
\begin{subequations}		\label{RF gauged system}
\begin{align}
\partial_t \check g 	&= -2\check{\Rc}+2\sum_a n_a \cv w_a \ten \cv w_a,			\label{RF base alt}\\
\partial_t w_a		&= \check\Delta w_a-\mu_a e^{-2w_a},						\label{RF fiber alt}
\end{align}
\end{subequations}
in which $\check\Delta$ is the Laplacian of the base.
\medskip

We find it convenient below to use both systems~\eqref{RF full system} and~\eqref{RF gauged system}.
Note that in either system, $\check g$ evolves but each $\hat g_a$ is fixed.

\subsection{Evolution of the scalar curvature of the base}

We assume in this subsection that the dimension of the base is strictly greater than one. We then derive the evolution equation
satisfied by $\check R$.

It is well known that if $\check g$ is any evolving Riemannian metric, then one has the variation formula
\[
\partial_t\check g=:\tilde h\qquad\Rightarrow\qquad
\partial_t \check R=-\check\Delta \tilde H+ \check\delta^2 \tilde h - \lp\check{\Rc},\tilde  h\rp_{\check g} =: Y_1+Y_2+Y_3,
\]
where $\tilde  H:=\check g^{ij}\tilde h_{ij}$ and $(\delta \tilde h)_j =-\check g^{ij}\cv_i \tilde h_{ij}$. Because $\check R$ is invariant under diffeomorphism,
we use form~\eqref{RF base alt} to choose
\[
\tilde h=-2\check{\Rc}+2\sum_a n_a v_a^{-2}\cv v_a\ten\cv v_a	\qquad\Rightarrow\qquad
\tilde H=-2\check R +2\sum_a n_a v_a^{-2}|\cv v_a|^2.
\]

We proceed to compute that
\begin{align*}
Y_1	&:=-\check\Delta \tilde H\\
	&= 2\check\Delta\check R\\
	&\quad-4\sum_a n_a \Big\{v_a^{-2}|\check\cv^2 v_a|^2+v_a^{-2}\lp\check\Delta\cv v_a,\cv v_a\rp
	-2v_a^{-3}\lp\cv v_a,\cv|\cv v_a|^2\rp\\
	&\qquad\qquad\qquad-v_a^{-3}|\cv v_a|^2\check\Delta v_a+3v_a^{-4}|\cv v_a|^4\Big\},\\
Y_2	&:= \check\delta^2 \tilde h\\
	&=-\check\Delta\check R\\
	&\quad+2\sum_a n_a\Big\{
	v_a^{-2}\big((\check\Delta v_a)^2+|\check\cv^2 v_a|^2+\lp\cv\check\Delta v_a,\cv v_a\rp + \lp\check\Delta\cv v_a,\cv v_a\rp\big)\\
	&\qquad\qquad\qquad-v_a^{-3}\big(4|\cv v_a|^2\check\Delta v_a+3\lp\cv v_a,\cv|\cv v_a|^2\rp\big) +6v_a^{-4}|\cv v_a|^4\Big\},\\
Y_3	&:=- \lp\check{\Rc},\tilde h\rp_{\check g} = 2|\check{\Rc}|_{\check g}^2-2\sum_a n_a v_a^{-2}\,\check{\Rc}(\cv v_a,\cv v_a).
\end{align*}
Collecting terms and commuting derivatives using the Ricci identities, the evolution equation for $\check R$ becomes
\begin{align*}
\partial_t \check R	&= \check\Delta\check R + 2|\check{\Rc}|^2-4\sum_a n_a v_a^{-2}\,\check{\Rc}(\cv v_a, \cv v_a)\\
	&\quad+2\sum_a n_a\Big\{v_a^{-2}\big[(\check\Delta v_a)^2-|\check\cv^2 v_a|^2\big]
	+v_a^{-3}\big[\lp\cv v_a,\cv|\cv v_a|^2\rp-2|\cv v_a|^2\check\Delta v_a\big]\Big\}.
\end{align*}
To simplify this formula, we write $w_a=\log v_a$ as above and write the preceding equation as:
\begin{equation}		\label{Base R evolution}
(\partial_t-\check\Delta)\check R	=   2|\check{\Rc}|^2-4\sum_a n_a \check{\Rc}(\cv w_a, \cv w_a)
			+2\sum_a n_a\Big\{(\check\Delta w_a)^2-|\check\cv^2w_a|^2\Big\}.
\end{equation}

\section{Classifying warped product shrinking solitons on $\mathbb{R}^2\times\mathbb{S}^2$}	\label{Classification}

The goal of this section is to prove Theorem \ref{thm-classification}. Its motivation is as follows. Elsewhere in this paper, we show (under mild hypotheses)
that a multiply warped product
$\big(\mc M = \mc B^d\times \mb S^{n_1}_1\times\mb S^{n_2}_2\times\cdots\times\mb S^{n_A}_A,g=\check g+\sum_a v_a^2\,\hat g_{\mb S^{n_a}_a}\big)$
for which one fiber (without loss of generality, $\mb S^{n_1}$) is sufficiently small, develops a finite-time singularity modeled by $\mc K\times\mb R^K$,
where $\mc K$ is a shrinking soliton of dimension $d+n_1$ and $K=n_2+\cdots +n_a$. In general, one does not know much about the structure of $\mc K$.
But if the base and the crushed fiber are both two-dimensional, one can say considerably more, which we now prove. The manifold considered in this section
is a shrinking (ancient) Ricci soliton with topology $\mc B^2\times\mb S^2$.

\subsection{Uhlenbeck frame}\label{sec-Uhlenbeck}
First, we obtain information about a singularity limit of a solution over $\mc B^2$ if the fiber that crushes is $\mb S^2$,
using Hamilton's interpretation of $\Rm$ as a symmetric bilinear operator on $\mathfrak{so}(4)=\mathfrak{so}(3)\oplus\mathfrak{so}(3)$ and
Uhlenbeck's trick~\cite{Ham4}.
Because the limit soliton $\mc K^4$ is a warped product, we are able to exploit a time-independent change of basis, an approach which is generally
infeasible in dimension four.

We assume here that $(e_1,e_2)$ is a (local) orthonormal frame field for  $T\mc B^2$ and $(e_3,e_4)$ is a (local) orthonormal frame field for $T\mb S^2$,
both of which are evolving by Uhlenbeck's trick to remain orthonormal.

We define
\[
	\lambda_1 := \check R,\quad 
	\lambda_2 := \frac{1-|\cv v|^2}{v^2},\quad 
	\lambda_3 := -\frac{\check\cv^2_{11}v}{v},\quad 
	\lambda_4 := -\frac{\check\cv^2_{22}v}{v},\quad 
	\lambda_5 := -2\frac{\check\cv^2_{12}v}{v}.
\]
We note the factor of $2$ in $\lambda_5$, which simplifies some formulas below.

We begin with a natural orthogonal basis $\beta$ of $\wedge^2 T\mc M$ given by
\begin{align*}
\beta_{1}  &  =e_{1}\wedge e_{2},\quad\beta_{2}=e_{1}\wedge e_{3},\quad\beta_{3}=e_{1}\wedge e_{4},\\
\beta_{4}  &  =e_{2}\wedge e_{3},\quad\beta_{5}=e_{2}\wedge e_{4},\quad\beta_{6}=e_{3}\wedge e_{4}.
\end{align*}
Using the formulas in Appendix~\ref{Curvature formulas} above and bearing in mind that $(e_1,e_2;\, e_3,e_4)$ is an orthonormal basis for $g$,
one computes the matrix of the curvature operator with respect to  the basis $\beta$. We obtain $M_{11}=R_{1221}+R_{2112}=2R_{1221}$, $M_{12}=2R_{1231}$,
$M_{13}=2R_{1241}$, and so forth, yielding the matrix
\[
M_\beta = \mat
	2\lambda_1 & 0 & 0 & 0 & 0 & 0 \\
	0 & 2\lambda_3 & 0 & \lambda_5 & 0 & 0\\
	0 & 0 & 2\lambda_3 & 0 & \lambda_5 & 0\\
	0 & \lambda_5 & 0 & 2\lambda_4 & 0 & 0\\
	0 & 0 & \lambda_5 & 0 & 2\lambda_4 & 0\\
	0 & 0 & 0 & 0 & 0  & 2\lambda_2\rix.
\]

We now use the Hodge-star decomposition $\Lambda^2 TM = \Lambda_+^2 \oplus \Lambda_-^2$, where 
\[
\Lambda_+^2=\{ \phi\in \Lambda^2 TM\,| \ast\phi=\phi  \},\quad
\Lambda_-^2 = \{ \psi\in \Lambda^2 TM\,| \ast\psi = -\psi  \}.
\]
We obtain orthonormal bases $(\phi_i)$ and $(\psi_i)$ for $\Lambda_+^2$ and $\Lambda_-^2$, respectively, as follows:
\[
	\sqrt{2}\phi_1 = e_1\wedge e_2 + e_3\wedge e_4,\quad
	\sqrt{2}\phi_2 = e_1\wedge e_3 + e_4\wedge e_2,\quad
	\sqrt{2}\phi_3 = e_1\wedge e_4 + e_2\wedge e_3,
\]
\[
	\sqrt{2}\psi_1 = e_1\wedge e_2 - e_3\wedge e_4,\quad
	\sqrt{2}\psi_2 = e_1\wedge e_3 - e_4\wedge e_2,\quad
	\sqrt{2}\psi_3 = e_1\wedge e_4 - e_2\wedge e_3.
\]

The basis $\alpha=(\phi_1,\phi_2,\phi_3,\psi_1,\psi_2,\psi_3)$ is given by the right action $\alpha=\beta \mc A$, where $\mc A$ is the orthogonal matrix
\[
\mc A=\frac{1}{\sqrt 2}\mat
1 & 0 & 0 & 1 & 0 & 0\\
0 & 1 & 0 & 0 & 1 & 0\\
0 & 0 & 1 & 0 & 0 & 1\\
0 & 0 & 1 & 0 & 0 & -1\\
0 & -1 & 0 & 0 & 1 & 0\\
1 & 0 & 0 & -1 & 0 & 0\\
\rix.
\]
The matrix of the curvature operator with respect to $\alpha$ is $M_\alpha = \mc A^{-1}M_\beta \mc A$.
One finds easily that $M_\alpha$ has the block structure
\[
M_\alpha = \mat A & B\\ B^T & C\rix,
	\quad\mbox{ where }\quad
	A = C = \mat  a_1 & 0 & 0\\
		0 & a_2 & 0\\ 
		0 & 0 & a_2\rix
\quad\mbox{ and }\quad
B = \mat b_1 & 0 & 0\\
		0 & b_2 & -\lambda_5\\ 
		0 & \lambda_5 & b_2\rix,
\]
which can be expressed in terms of
\[
	a_1 :=\lambda_1+\lambda_2,\quad
	a_2 := \lambda_3+\lambda_4,\qquad
	b_1 := \lambda_1-\lambda_2,\quad
	b_2 :=\lambda_3-\lambda_4.
\]

As explained in~\cite{Ham4}, Uhlenbeck's trick leads to the following system:
\begin{align*}
	(\partial_t-\Delta) A &= A^2+2A^\# +BB^T,\\ 
	(\partial_t-\Delta) B &= AB+BC +2B^\#,\\ 
	(\partial_t-\Delta) C &= C^2+2C^\# +B^TB.
\end{align*}
Then using standard matrix operations, we compute that the components of $M_\alpha$ evolve by:
\begin{align*}
(\partial_t-\Delta) a_1
&= a_1^2+2a_2^2 + b_1^2,\\ 
(\partial_t-\Delta) a_2
&= a_2^2+2a_1a_2 + b_2^2+\lambda_5^2,\\ 
(\partial_t-\Delta) b_1
&= 2a_1b_1 + 2b_2^2+2\lambda_5^2,\\ 
(\partial_t-\Delta) b_2
&= 2a_2b_2 +2b_1b_2=2(a_2+b_1)b_2,\\
(\partial_t-\Delta) \lambda_5
&= 2a_2\lambda_5 + 2b_1\lambda_5
=2(a_2+b_1)\lambda_5.\\
\end{align*}
It is important to note that the warped product structure means that these are truly partial differential equations,
rather than the ordinary differential inequalities one is usually forced to work with in dimension four. This is because
we are able to make a single time-independent choice of bases.
\medskip

The considerations above yield the following:

\begin{lemma}
\label{lemma-base-R2}

If $(M^4=\mathcal{B}^2\times \mathbb{S}^2,g_t)_{t\le 0}$ is an ancient Ricci flow that is a warped product with complete time slices,
then $\check R\geq0$. 

\end{lemma}
\begin{proof}
We observe that $a_1+b_1=2\check R$ and that
\begin{equation}
\label{eq-R-uhlen}
	(\partial_t-\Delta)(a_1+b_1) = (a_1+b_1)^2+2(a_2^2+b_2^2+\lambda_5^2) \ge (a_1+b_1)^2.
\end{equation}
The set $a_1+b_1=2\check R\geq0$ is closed and convex in the space of symmetric bilinear forms on $\mathfrak{so}(4)$ and is invariant under parallel translation.
So the tensor maximum principle applies and lets us conclude that $\check R\geq0$ on the ancient limiting soliton by using Bing-Long Chen's argument~\cite{BLC09}.

\end{proof}

Before we prove Theorem \ref{thm-classification}, we  recall some identities that hold on gradient shrinking Ricci solitons. Using those we derive
new identities involving the warping function, that play an important role in proving the classification result above.

\subsection{Standard formulas}

The gradient Ricci soliton equation is
\[
\Rc[g] + \lambda g + \cv^2 f = 0,
\]
where $\lambda<0$ on a shrinker. By scaling we may assume  $\lambda=-\frac12$, yielding
\begin{align}
\Rc+\cv^2f-\frac{1}{2}g	&=0,\\
\Delta(\mr df) + \Rc(\mr df)	&=0,\\
R+|\cv f|^2-f	&=c_n,				\label{normalization}
\end{align}
where $c_n$ is a constant that can be determined with extra hypotheses.

\subsection{Warped product geometry}		\label{Dimension-2}
For $i,j,k,\ell\in\{1,2\}$ and $\alpha,\beta,\gamma,\delta\in\{1,\dots,n\}$, we write $g_{\alpha\beta}=v^2\,\hat g_{\alpha\beta}$ and observe that
the geometric curvature data of $g$ are determined by
\begin{align}
R_{ijk\ell}	&=\check R_{ijk\ell},\\
R_{\alpha\beta\gamma\delta}	&= v^2\hat R_{\alpha\beta\gamma\delta}-\frac{|\cv v|^2}{v^2}\big(g_{\alpha\delta}g_{\beta\gamma}-g_{\alpha\gamma}g_{\beta\delta}\big),\\
R_{i\alpha\beta j}	&=-v^{-1}(\check\cv^2 v)_{ij}\,g_{\alpha\beta},\\
R_{ij}	&=\frac12 \check R\,\check g_{ij}-nv^{-1}(\check\cv^2 v)_{ij},\\
R_{\alpha\beta}	&=\Big\{(n-1)\frac{1-|\cv v|^2}{v^2}- v^{-1}\check\Delta v\Big\}g_{\alpha\beta},\\
R&=\check R-2n\frac{\check\Delta v}{v}+n(n-1)\frac{1-|\cv v|^2}{v^2}.			\label{R}
\end{align}

These formulas may be be compared to Proposition~9.106 of~\cite{Besse}. 

\subsection{Warped product shrinkers}
Consider a shrinking soliton limit $\tilde{\mc B}^2\times_v \mc F^n$. Then recalling that $N=n+2$, one defines
\begin{align*}
\Rc_f	&:=\Rc+\cv^2f,\\
\check\Delta_f\zeta	&:=\check\Delta\zeta-\lp\cv f,\cv\zeta\rp,	&(\forall\zeta:\mc B^2\to\mb R),\\
\mr d\nu	&:= (4\pi)^{-N/2}e^{-f}\,\mr dV,\\
\mr d\check\nu	&:= (4\pi)^{-N/2}e^{-f}|\mc F^n|\,\mr dA,
\end{align*}
noting that $\check\Delta_f$ is self-adjoint with respect to either measure above. Then one has
\begin{align}
\Rc_f	&=\frac12 g,	&(\mbox{\blue full:soliton}),	\label{full:soliton} \\
\frac12(\check R-1)\check g_{ij}	&=nv^{-1}\check\cv^2_{ij} v - \check\cv^2_{ij}f,	&(\mbox{\blue base:tensor}),	\label{BT}\\
\check R&= 1+nv^{-1}\check\Delta v-\check\Delta f,		&(\mbox{\blue base:scalar}),			\label{BS}\\
v^{-1}\check\Delta_fv &=(n-1)v^{-2}(1-|\cv v|^2)-\frac12,		&(\mbox{\blue fiber:scalar}),		\label{FS} \\
\int_{\mc M}\zeta\,\mr d\nu		&= \int_{\mc B}\zeta\,v^n\,\mr d\check\nu,	&(\forall\zeta:\mc B^2\to\mb R).
\end{align}
We are most interested in the formulas above in the case $n = 2$.

\subsection{A new identity}
We trace~\eqref{BT} with $\check g^{-1}$ to get~\eqref{BS}, which we write here as
\begin{equation}	\label{R-1 Id}
(\check R-1) = \frac{2}{v}\check\Delta v-\check\Delta f.
\end{equation}
Next, we contract~\eqref{BT} with $\check\cv^2 v$ to get
\begin{equation}	\label{Id1}
\frac12(\check R-1)\check\Delta v = \frac{2}{v}|\check\cv^2 v|^2 - \lp \check\cv^2 v,\check\cv^2 f\rp,
\end{equation}
and then contract~\eqref{BT} with $\check\cv^2 f$ to get
\begin{equation}	\label{Id2}
\frac12(\check R-1)\check\Delta f = \frac{2}{v} \lp \check\cv^2 v,\check\cv^2 f\rp - |\check\cv^2 f|^2.
\end{equation}

Now, we add~\eqref{Id1} and $\frac{v}{n}$ times~\eqref{Id2} and substitute~\eqref{R-1 Id} to obtain, after a wee bit of algebra,
\begin{equation}	\label{Nice Id}
\frac{2}{v}\Big\{|\check\cv^2 v|^2-\frac12(\check\Delta v)^2\Big\} = \frac{v}{2}\Big\{|\check\cv^2 f|^2-\frac12(\check\Delta f)^2\Big\}.
\end{equation}

This can be rewritten in a slicker way. Let $A=\check\cv^2 v$ and $B = \check\cv^2 f$, and denote their trace-free parts
by $\stackrel{\circ} A$ and $\stackrel{\circ}B$, respectively. Then, whatever the dimension $n$ of the fibers may be, equation~\eqref{Nice Id} is equivalent to
\begin{equation}	\label{Nicer Id}
\big|\stackrel{\circ}A\big|^2 = \big(\frac{v}{2}\big)^2 \big|\stackrel{\circ}B\big|^2.
\end{equation}
\medskip

For use in the proofs of the next Lemma and Theorem~\ref{thm-classification}, we define
\begin{equation}	\label{define_h}
h:=b_2^2+\lambda_5^2.
\end{equation}

\begin{lemma}
\label{lemma-useful}
Let $(\mathbb{R}^2\times\mathbb{S}^2, g)$ be a nonflat shrinking Ricci soliton as above. Then the identities
\[a_2 + \lambda_2  = \frac12, \qquad a_2 b_1 = 0, \qquad \mbox{and} \qquad \langle\nabla f, \nabla v\rangle = 0,\]
hold everwhere on $\mathbb{R}^2\times\mathbb{S}^2$, where $a_1, a_2, b_1, \lambda_5$ are defined in Section~\ref{sec-Uhlenbeck}.
\end{lemma}

\begin{proof}
On a shrinker in our setting, by~\eqref{BS} and~\eqref{FS}, we have
\[
    \check R -1 = 2v^{-1}\check\Delta v - \check\Delta f,\qquad 
    v^{-1}\check \Delta_f v = \lambda_2 - \tfrac{1}{2}.
\]
Plugging the formula for $2v^{-1}\check\Delta v$ into the second equation, one gets
\[
    \check R - 1 + \check \Delta f - 2 v^{-1}\nabla v\cdot \nabla f
    = 2\lambda_2 - 1,
\]
which is the same as
\[
    \check R - 2\lambda_2 = -\Delta f.
\]
By taking the trace of \eqref{full:soliton} we obtain $R + \Delta f = 2$, and hence by \eqref{R}, we have
\[
    \check R - 2\lambda_2 = R-2 = \check R + 4a_2 +2\lambda_2 - 2,
\]
and thus on the static shrinker, or at time $-1$ of the flow,  we have
\[  a_2 + \lambda_2 = \tfrac{1}{2},
\]
as claimed. The third identity, $\langle\nabla f, \nabla v\rangle = 0$ follows immediately by adding \eqref{FS}  to the previous identity.

To prove the final identity, recall that if $Q$ scales like $R$ on a shrinker, then
\[
    \Box Q = Q - \Delta_f Q,
\]
where, here and below, $\Box=\frac{\partial}{\partial t}-\Delta$.

Also recall that $\Box \lambda_2 = \frac{1}{2}\Box(a_1-b_1)=2\lambda_2^2+a_2^2-h$, where $h$ is defined in~\eqref{define_h}.
Thus we have
\[
\begin{split}
    \tfrac{1}{2} &= a_2+\lambda_2 - \Delta_f(a_2+\lambda_2) = \Box (a_2+\lambda_2)\\ 
    &= a_2^2+2a_1a_2+h + 2\lambda_2^2+a_2^2-h \\ 
    &= 2(a_2+\lambda_2)^2 + 2a_2(a_1-2\lambda_2) \\ 
    &= \tfrac{1}{2} + 2a_2 b_1.
\end{split}
\]
This implies that
\[ a_2 b_1 \equiv 0,
\]
concluding the proof of the Lemma.
\end{proof}

We are now ready to achieve the goal of this section.

\begin{proof}[Proof of Theorem~\ref{thm-classification}]
By the strong maximum principle, either  $\check R>0$ everywhere or $\check{R} \equiv 0$ on {$\mc B^2\times\mb S^2$}.
We show below that the former case cannot happen.
\medskip
 
 Assume first that $\check{R} \equiv 0$ on  {$\mc B^2\times\mb S^2$}. 
 Recall that by \eqref{eq-R-uhlen}, we have
\begin{equation}\label{eq-R-2}
    \Box \check R = 2\check R^2 + a_2^2+h,
\end{equation}
where $h$ is defined in~\eqref{define_h}.
Since $\check{R} \equiv 0$, \eqref{eq-R-2} implies that $a_2 \equiv 0$, $b_2\equiv 0$, and $\lambda_5 \equiv 0$.
This implies that {$\check\nabla^2 v = 0$ on $\mc B^2$.}
 By \eqref{BT}, $\check\nabla^2 f = \frac{1}{2}\check g,$ and thus $(\mc B^2,\check g)$ must be isomorphic to $\mb R^2$ with the Euclidean metric
(see, \emph{e.g.,} page 218 of \cite{DZ}). 
Then $v = ax_1 + bx_2 + c$, for some constants $a, b, c \in \mathbb{R}$, and all $x_1, x_2\in \mathbb{R^2}$.
Since $v > 0$, this requires that $a = b = 0$, and hence $v \equiv c$ on $\mathbb{R}^2$.
This yields the conclusion of Theorem~\ref{thm-classification}.
\medskip

Assume now that $\check{R} > 0$ everywhere on $\mc B^2\times\mb{S}^2$. Then $\mc B^2$ is diffeomorphic to $\mb R^2.$ Our goal is to show that this is not possible. In order to achieve this,
we define $G := \frac{\sqrt{h}}{{\check{R}}}$. We claim  that $h \equiv 0$ on $\mathbb{R}^2\times\mathbb{S}^2$, and to show this we first compute its evolution equation.

By computations in Section~\ref{sec-Uhlenbeck}, we have
\[
\Box h = 4(a_2+b_1)h - 2|\nabla b_2|^2 - 2|\cv\lambda_5|^2.
\]
It follows that
\[
    \Box \sqrt{h} = \tfrac{1}{2}h^{-\frac{1}{2}} \Box h + \tfrac{1}{4}h^{-\frac{3}{2}}|\nabla h|^2
    = 2(a_2+b_1)\sqrt{h} -(|\nabla b_2|^2+|\nabla \lambda_5|^2)h^{-\frac{1}{2}} + \tfrac{1}{4}h^{-\frac{3}{2}}|\nabla h|^2.
\]
Because
{
\begin{align*}
  \tfrac{1}{4}|\nabla h|^2 &= |b_2\nabla b_2 + \lambda_5\nabla \lambda_5|^2
    = b_2^2|\nabla b_2|^2+\lambda_5^2|\nabla \lambda_5|^2 + 2b_2\lambda_5\lp\nabla b_2,\nabla \lambda_5\rp\\ 
    &\le   b_2^2|\nabla b_2|^2+\lambda_5^2|\nabla \lambda_5|^2 +\lambda_5^2|\nabla b_2|^2 + b_2^2|\nabla\lambda_5|^2\\ 
    &= (b_2^2+\lambda_5^2)(|\nabla b_2|^2+|\nabla \lambda_5|^2),
\end{align*}}
we have
\[
    \Box \sqrt{h}  \le 2(a_2+b_1)\sqrt{h}.
\]
It follows that
\begin{align*}
\Box G&= \frac{\Box\sqrt{h}}{\check R} - \frac{G}{\check R}\Box\check R+2\lp\nabla G, \nabla\log\check R\rp\\ 
&\le 2(b_1+a_2)G - \frac{G}{\check R}(2\check R^2+a_2^2+h)  + 2\lp\nabla G,\nabla\log\check R\rp \\
&=  (2b_1+2a_2 - 2\check R - a_2^2/\check R)\,G - \check RG^3 +  2\lp\nabla G,\nabla\log\check R\rp. 
\end{align*}

We now define $P:=  2b_1+2a_2 - 2\check R - a_2^2/\check R.$ We claim that $P < 0$ on $\mathbb{R}^2\times\mathbb{S}^2$.
Indeed, at any point of our manifold, by Lemma \ref{lemma-useful}, there are two cases:

\underline{Case 1}: $a_2=0.$ Then $\lambda_2 = \frac{1}{2}$ and
\[
P=2b_1+2a_2 - 2\check R - a_2^2/\check R
= 2b_1-2\check R = -2\lambda_2 = -1< 0.
\]

\underline{Case 2}: $b_1=0.$ Then
\[
P=2a_2 - 2\check R - a_2^2/\check R \le -\check R < 0
\]
by the Cauchy--Schwarz inequality.
\medskip

Recall that $f$ is the shrinker potential. Since $G$ is a scaling invariant quantity, at time $t=-1,$ by the computations above, on the soliton
$(\mb R^2\times \mb S^2, g)$ with a warped product structure, we have
\begin{align*}
 -\Delta_f G &=\Box G\le PG-\check R G^3+ 2\lp\nabla G, \nabla\log\check R\rp.
\end{align*}
 Let $\tilde f = f -2\log\check R.$ Then we have
\[
    -\Delta_{\tilde f}G \le PG-\check R G^3.
\]
Let $\eta=\eta_r$ be a standard cutoff function with $\eta_r|_{B_r(o)}=1,|\nabla \eta_r|\le 10/r.$ Multiplying the inequality above by $\eta^2 G e^{-\tilde f}$
and integrating it over the entire manifold $M := \mb R^2\times \mb S^2$ yields
\[
    -\int_{M} \eta^2 G \Delta_{\tilde f}Ge^{-\tilde f} \,\mr dV_g
    \le \int_{M} \eta^2(PG^2-\check R G^4) e^{-\tilde f}\, \mr dV_g,
\]
where the left-hand side after integration by parts can be rewritten as
\begin{align*}
    \int_{M} \eta^2 |\nabla G|^2 e^{-\tilde f}\, \mr dV_g &+ 2\int_{M} \eta G\nabla \eta \cdot \nabla G e^{-\tilde f}\,\mr dV_g  \\ 
    \ge &\ \frac{1}{2}\int_{M} \eta^2 |\nabla G|^2 e^{-\tilde f}\,\mr dV_g  - 2\int_{M}  |\nabla \eta|^2 G^2  e^{-\tilde f}\,\mr dV_g \\ 
    \ge & \ \frac{1}{2}\int_{M} \eta^2 |\nabla G|^2 e^{-\tilde f}\,\mr dV_g  - \frac{C}{r^2}\int_{M} G^2  e^{-\tilde f}\,\mr dV_g. 
\end{align*}
By \eqref{Nice Id}, we have
\[
    h = 2|-v^{-1}\check \nabla v + \tfrac{1}{2} v^{-1}\check\Delta v\check g|^2 = \tfrac{1}{2}|\check\nabla^2 f - \tfrac{1}{2}\check\Delta f\check g|^2
    \le 2|\nabla^2 f|^2.
\]
It follows that
\begin{align}\label{eq-int-bound}
 \int_{M} G^2  e^{-\tilde f}\,\mr dV_g 
 &=     \int_{M} \check R^2 G^2  e^{-f}\,\mr dV_g  = \int_{M} h e^{-f}\,\mr dV_g \\
 &\le 2\int_{M} |\nabla^2 f|^2e^{-f}\,\mr dV_g 
= 2\int_{M} |\tfrac{1}{2}g - \Rc|^2 e^{-f}\,\mr dV_g   \le C, \nonumber
\end{align}
where in the last inequality,  we use~\cite{MS} to see that $\int_{\mb R^2\times\mb S^2} |\Rc|^2 e^{-f}\, \mr dV_g \le C$, and the result from \cite{CZ} that the soliton potential has quadratic growth. Combining the estimates above, we obtain
\[
    \int_{\mb R^2} \eta^2|\nabla G|^2 e^{-\tilde f}\,\mr dV_g 
    \le \int_{\mb R^2} \eta^2(PG^2-\check RG^4) e^{-\tilde f}\,\mr dV_g  + C/r^2. 
\]
Taking $r\to \infty$  in the inequality above, using the fact that $P \le 0$, and bearing in mind that~\eqref{eq-int-bound} holds, we obtain
\[
    \int_{\mb R^2} |\nabla G|^2e^{-\tilde f}\,\mr dV_g   = 0.
\]
Hence, $G\equiv 0$ and $b_2=\lambda_5=0.$ By \eqref{Nice Id}, this implies $\check\nabla^2 f=\mu \check g$ for some function $\mu$, where $f$ is the soliton potential function.  By Tashiro's theorem~\cite{Tashiro}, the base itself is a warped product $\check g = \mr dr^2+\rho^2(r)\,\mr d\theta^2$, 
where $\rho(r)=|\nabla f|/a=f'(r)/a>0$ for some constant $a>0$ such that $\rho'(0)=1$. Furthermore, since $h\equiv 0$, we have
\[
    v''\mr dr^2+\rho\rho'v'\mr d\theta^2=\check\nabla^2 v = \tfrac{1}{2}\check\Delta v\check g
    = \tfrac{1}{2}(v''+\rho^{-1}\rho'v')\check g.
\]
This implies that
\[
    v'' = \rho^{-1}\rho'v',\qquad \iff\qquad 
    \rho v'' = \rho'v'.
\]
Then $(v'/\rho)'=0$ and $v'=b f'(r)$ for some constant $b$. By Lemma \ref{lemma-useful},
we have $f'(r) v'(r) \equiv 0$. Together, these imply that $b = 0$, and hence that $v'(r) \equiv 0$, which implies $v \equiv c$ on $\mb R^2$ for some constant $c$.
This together with
$a_2 + \lambda_2 = \frac12$ implies further that that $v^2 \equiv 2$ on $\mb R^2$. \
By examining the curvature representation discussed in Section~\ref{sec-Uhlenbeck}, we see that   our soliton $(\mb R^2\times \mb S^2, g)$
has nonnegative isotropic curvature and thus, {by~\cite[Corollary 3.1]{LNW}}, is isometric to the bubble sheet soliton metric on $\mb R^2\times \mb S^2$.
This in particular implies that $\check{R} \equiv 0$, showing that the case $\check{R} > 0$ everywhere on $\mb R^2$ cannot occur, as claimed above.

This concludes the proof of Theorem~\ref{thm-classification}.

\end{proof}

\section{General elementary estimates}	\label{Elementary estimates}

In this section, we establish a few preliminary $C^k\;(k=0,1,2)$ estimates that are useful for our applications below. 

\subsection{$C^0$ estimates}

\begin{lemma}		\label{C0 estimates}
For as long as a solution exists, the size of each fiber with $\mu_a\geq0$ satisfies
\[
(v_a)_{\max}(t)\leq (v_a)_{\max}(0)\qquad\mbox{ and }\qquad (v_a)_{\min}(t)\leq\sqrt{2\mu_a(T_a-t)}
\]
if there exists a first time $T_a>0$ such that $(v_a)_{\min}=0$.
\end{lemma}

\begin{proof}
Equation~\eqref{RF fibers} implies that $v$ is a subsolution of the heat equation, which proves the first claim.

If $T_a$ exists, the second claim follows from $\frac{\mr d}{\mr dt}(v_a^2)_{\min}(t)\geq-2\mu_a$ by integration from $t$ to $T_a$.
\end{proof}

Without loss of generality, we relabel if necessary so that $\mc F_1$ is crushed first.

\begin{assumpA}		\label{single fiber pinching}
The  initial data satisfies the assumption of \emph{single fiber pinching} if
\[
\frac{(v_{a}^2)_{\min}(0)}{2\mu_{a}} \ge \frac{(v_1)_{\max}^2(0)}{\mu_1}
\]
for all $a\in \{2,3, \dots, A\}$. 
\end{assumpA}

The reason for the name above is the following:

\begin{lemma}		\label{lemma-initial}
Consider a multiply warped product over any compact base $\mc B^n$ that satisfies Assumption~\ref{single fiber pinching}.
Then there exist a time $T < \infty$ and a constant $\delta > 0$ so that
$\liminf_{t\to T}\big\{\min_{x\in \mc B^n} v_1(x,t)\big\} = 0$, but $\min_{x\in \mc B^n} \big\{v_a(s,t)\big\} \ge \delta > 0$ for all $t\in [0,T)$ and all $a\in \{2,3,\dots, A\}$.
\end{lemma}

\begin{proof}
By the maximum principle applied to the evolution equation~\eqref{RF fibers} satisfied by $v_a$, we have
\[
\frac{d}{dt} (v_{a}^2)_{\min} \ge -2\mu_{a},
\]
for $a\in \{2,3,\dots,A\}$, yielding
\begin{equation}			\label{eq-lower-v}
(v_a^2)_{\min}(t) \ge (v_a)_{\min}^2(0) -2 \mu_a t = 2\mu_a\, \Big(\frac{(v_a^2)_{\min}(0)}{2\mu_a} - t \Big).
\end{equation}

On the other hand, we also have
\[
\frac{d}{dt}(v_1^2)_{\max} \le- 2\,\mu_1,
\]
implying that
\[
(v_1)_{\max}^2(t) \le 2\mu_1\,\Big(\frac{(v_1^2)_{\max}(0)}{2\mu_1} - t\Big).
\]
This implies that there exists a finite time
\[
T \le \frac{(v_1^2)_{\max}(0)}{2\mu_1} < \infty\qquad\mbox{ at which }\qquad
\liminf_{t\to T} \big\{ \min_{s\in\mb S^1} v_1(s,t)\big\} = 0.
\]

Finally, Assumption~\ref {single fiber pinching} and inequality~\eqref{eq-lower-v}  together imply that for all $t \in [0,T)$ and all $a\in \{2,3, \dots, A\}$, we have
\[
(v_{a}^2)_{\min}(t) \ge 2\mu_{a}\, \Big(\frac{(v_{a}^2)_{\min}(0)}{2\mu_{a}} - T\Big) \ge 2\mu_{a}\, \Big(2\frac{(v_1^2)_{\max}(0)}{2\mu_1} - T\Big)
	\ge \frac{\mu_{a}}{\mu_1}\,(v_1^2)_{\max}(0) \ge \delta,
\]
where $\delta := \min_{a\in \{2,\dots, A\}} \frac{\mu_a}{\mu_1} (v_1^2)_{\max}(0)>0$. 
\end{proof}

\subsection{A scale-invariant $C^1$ estimate}

\begin{lemma}		\label{C1 estimate}
If $n_a\geq2$, then for as long as a solution exists,
\[
|\cv v_a|^2_{\max}(t)\leq\max\Big\{ |\cv v_a|^2_{\max}(0), \;\frac{\mu_a}{n_a-1}\Big\}.
\]
\end{lemma}

\begin{proof}
Consulting Appendix C of~\cite{CIKS22} and using the consequence of~\eqref{Tensor norm} that
\[
|\cv^2v_a|^2\geq|\check\cv^2v_a|_{\check g}^2+n_a\frac{|\cv v_a|^4}{v_a^2},
\]
we find that
\begin{align*}
\partial_t (|\cv v_a|^2)	&=\Delta(|\cv v_a|^2)-\frac12v_a^{-2}|\cv^2(v_a^2)|^2+v_a^{-2}|\cv v_a|^2(2\mu_a+4|\cv v_a|^2)\\
	&=\Delta(|\cv v_a|^2)-2|\cv^2v_a|^2-4\frac{\cv^2v_a(\cv v_a,\cv v_a)}{v_a}+2\frac{|\cv v_a|^2}{v_a^2}(\mu_a+|\cv v_a|^2)\\
			&\leq\Delta(|\cv v_a|^2)-2|\check\cv^2v_a|_{\check g}^2-2\frac{\lp\cv|\cv v_a|^2,\cv v_a\rp}{v_a}
			+2\frac{|\cv v_a|^2}{v_a^2}\big(\mu_a-(n_a-1)|\cv v_a|^2\big),
\end{align*}
which shows that $|\cv v_a|^2$ cannot increase at a maximum if $|\cv v_a|^2\geq\mu_a/(n_a-1)$.
\end{proof}

\subsection{The curvature of $(\mc B^2,\check g)$}
In order to derive $C^2$ estimates for the warping functions, we need an estimate for the curvature of the base.
Here, we derive such an estimate for a base surface. The proof simplifies if one uses the gauged system~\eqref{RF gauged system}.

If the base is two-dimensional, equation~\eqref{Base R evolution} becomes
\begin{equation}	\label{R_t on surface}
\big(\partial_t-\check\Delta\big)\check R = \check R^2-2\check R\sum_a n_a|\cv w_a|^2
	+2\sum_an_a\big\{(\check\Delta w_a)^2-|\check\cv^2 w_a|^2\big\},
\end{equation}
where we recall that $w_a=\log v_a$.
We define
\[
p:=\sum_a n_a|\cv w_a|^2 \qquad\mbox{ and }\qquad  f:=\check R+2 p,
\]
so that we may use  $f$ as an upper bound for $\check R$. 
\medskip

We define an initial metric on a warped product with a two-dimensional base to be \emph{$\eta$-tame} if $f_{\max}(0)\leq\eta$.

\begin{assumpA}		\label{eta tame}
Our initial data is $\eta$-tame for $0<\eta\ll1$, to be determined.
\end{assumpA}

Note that we can always satisfy Assumption~\ref{eta tame} simply by taking a sufficiently large homothetic dilation of $\big(\mc B^2,\check g(0)\big)$.
\medskip

We begin by deriving an upper bound.

\begin{lemma} \label{lemma-checkR-upper}
If the base is two-dimensional and Assumptions~\ref{single fiber pinching} and~\ref{eta tame} hold for some $\eta$ sufficiently small,
then there exists $C_0$ depending only on the initial data such that for as long as a smooth solution exists,
\[
\check R\leq\max\Big\{C_0, \;\frac{2\mu_1}{3v_1^2}\Big\}.
\]
\end{lemma}

\begin{proof}
By~\eqref{RF fiber alt}, one has $\big(\partial_t-\check\Delta\big)w_a	= -\mu_a v_a^{-2}$,
and by the Bochner formula, one has
\begin{align*}
\big(\partial_t-\check\Delta\big)p	&=\sum_a n_a
\Big\{-2|\check\cv^2 w_a|^2-2\check\Rc(\cv w_a,\cv w_a)+4\mu_a v_a^{-2}|\cv w_a|^2\\
	&\qquad\qquad + 2\check\Rc(\cv w_a,\cv w_a) - 2\sum_b n_b
	\lp \cv w_a,\cv w_b \rp^2 \Big\}\\ 
&= -2\sum_a n_a |\check\cv^2 w_a|^2
+ 4\sum_a n_a\mu_av_a^{-2}|\cv w_a|^2
- 2\sum_{a,b}n_an_b \lp \cv w_a,\cv w_b \rp^2.
\end{align*}

By Lemma~\ref{lemma-initial}, there exist $T$ and $\delta$ such that $(v_1)_{\min}\searrow0$ as $t\nearrow T$,
but $(v_a)_{\min}\geq\delta$ as $t\nearrow T$ for all $a\in\{2,\dots,A\}$. Thus by the Cauchy--Schwarz inequality, we have
\begin{align*}
\big(\partial_t-\check\Delta\big)f	&= (\check R^2-2\check Rp)
+2\sum_an_a\big\{(\check\Delta w_a)^2-|\check\cv^2 w_a|^2\big\}\\ 
&\quad -2\sum_a n_a |\check\cv^2 w_a|^2
+ 4\sum_a n_a\mu_av_a^{-2}|\cv w_a|^2
- 2\sum_{a,b}n_an_b \lp \cv w_a,\cv w_b \rp^2\\
&\leq \check R^2 -2\check R p   + 4n_1 \mu_1 v_1^{-2}|\cv w_1|^2 + C_\delta\\
&= (f-2p)^2-2(f-2p)p + 4n_1 \mu_1 v_1^{-2}|\cv w_1|^2 + C_\delta\\
 &= f^2+8p^2-6fp + 4n_1 \mu_1 v_1^{-2}|\cv w_1|^2 + C_\delta\\
 &\le 2 f^2+20p^2 +C_\delta+2n_1(2\mu_1-3fv_1^2)\,v_1^{-2}|\cv w_1|^2,
 \end{align*}
where $C_\delta$ is independent of $T$.

Here, we make the following observation: if $f\ge \frac{2\mu_1}{3v_1^2}$ at any point, then by Assumption~\ref{single fiber pinching} and Lemma~\ref{C1 estimate},
the inequality
\begin{equation}		\label{f majorizes h}
	0<p \le C_0 v_1^{-2} \le \frac{3C_1}{2\mu_1}f
\end{equation}
holds at that point.

We now let $\phi$ solve the \textsc{ode} $\phi'(t)=C_2\phi^2(t)+C_\delta$ with the initial value $\phi(0)=f_{\max}(0)$, where $C_2=2+45C_1^2/\mu_1^2$.
By Assumption~\ref{eta tame}, we may assume that $\phi(T)\leq C_0$. Then there are two cases:
\begin{enumerate}

\item If $2\mu_1-3fv_1^2\leq0$ wherever $f_{\max}(t)$ is attained, then by estimate~\eqref{f majorizes h} above, we have
\[
\big(\partial_t-\check\Delta\big)f
\le C_2f^2 +C_\delta.
\]
Then the maximum principle implies that $f_{\max}(t)\leq\phi(t)\leq\phi(T)$ stays bounded.

\item  If $2\mu_1-3fv_1^2>0$ at any point that $f_{\max}(t)$ is attained, then we have $f_{\max} \le \frac{2}{3} \mu_1v_1^{-2}$ at that point.

\end{enumerate}
Because $\check R\leq f$, this concludes the proof of the lemma.
\end{proof}

An analogous lower bound is even easier to obtain, albeit at the cost of a larger constant.
To serve as a lower bound for $\check R$, we  define $\tilde f:=\check R-p$.

\begin{lemma} 	\label{lemma-checkR-lower}
If the base is two-dimensional and Assumption~\ref{single fiber pinching} holds,
there exists $C_1$ depending only on the initial data such that for as long as a smooth solution exists,
\[
\check R\geq-\frac{C_1}{v_1^2}.
\]
\end{lemma}

\begin{proof}
As in the proof of Lemma~\ref{lemma-checkR-upper}, one finds by using Lemma~\ref{lemma-initial} and Lemma~\ref{C1 estimate} that
\[
\big(\partial_t-\check\Delta\big)\tilde f	 \geq \tilde f^2-p^2 -  4n_1\mu_1v_1^{-2}|\cv w_1|^2-C_\delta
	\geq\tilde f^2 - C'_\delta v_1^{-4},
\]
where $C_\delta,\,C'_\delta$ depend only on the initial data. The \textsc{rhs} is positive if $f<-\sqrt{C'_\delta}v_1^{-2}$.
The result follows because $\check R\geq\tilde f$.
\end{proof}

\subsection{$C^2$ estimates}

To state our $C^2$ bound for $v_a$, for each $a\in\{1,\dots,A\}$, we define
\[
\chi_a:=|\cv^2 v_a|^2\quad\mbox{ which is bounded above by }\quad F:=\sum_b (B+|\cv v_b|^2)\chi_b,
\]
where $B>1$ is a large positive constant to be determined below.

In Appendix~\ref{Hessian evolution}, we derive that the norm of each covariant Hessian evolves as follows.
(Here and below, we use subscripted variables to denote partial derivatives.)

\begin{lemma}		\label{alt chi evolution}
Each quantity $\chi_a=|\cv^2 v_a|^2$ evolves by
\begin{align*}
(\chi_a)_t	&= \Delta\chi_a - 2|\cv^3 v_a|^2 - 2\,\boxed{\left\lp\cv^2 v_a,  \cv^2\Big(\frac{\mu_a+|\cv v_a|^2}{v_a}\Big)\right\rp}\\
	&\quad +4\check{\Rm}(\check\cv^2v_a,\check\cv^2v_a) -4\sum_b n_b v_b^{-2}\lp\cv v_a,\cv v_b\rp\lp\check\cv^2 v_a,\check\cv^2 v_b\rp_{\check g}\\
	&\quad +4\sum_b n_b\mu_b v_b^{-4}\lp\cv v_a,\cv v_b\rp^2
		-4\sum_b n_b(n_b-1)v_b^{-4}|\cv v_b|^2\lp\cv v_a,\cv v_b\rp^2\\
	&\quad -4\sum_b \sum_{c\neq b}
		n_b  n_c\,v_b^{-2} v_c^{-2} \lp\cv v_a,\cv v_b\rp \lp\cv v_b,\cv v_c\rp \lp\cv v_c,\cv v_a\rp,
\end{align*}
where the reaction term is boxed above.
\end{lemma}

\begin{theorem}		\label{Hessian bound}
Suppose that the base is one-dimensional, or else that the base is two-dimensional and Assumptions~\ref{single fiber pinching} and~\ref{eta tame} hold for some 
$\eta$ small enough so that Lemma~\ref{lemma-checkR-upper} applies.

Then there exist positive constants $C_*,C^*$ depending only on the initial data such that:
\[
F_t\leq\Delta F+C_*\sum_a v_a^{-4}\qquad\Rightarrow\qquad
\frac{\mr d}{\mr dt}F_{\max}\leq C^*\big(\min\{v_a\}\big)^{-4}.
\]
In particular, the flow exists until the first time $T>0$ that some $v_a=0$.
\end{theorem}

\begin{proof}
In the estimates below, $C$ and $\ve$ denote uniform large and small constants, respectively.
We allow $C,\ve$ to change from line to line without introducing circular dependencies.

The first part of the reaction term (boxed above) is easily estimated,
\[
\mu_a\Big|\big\lp\cv^2v_a,\cv^2(v_a^{-1})\big\rp\Big|\leq \mu_a\frac{\chi_a}{v_a^2}+2\mu_a\frac{|\cv v_a|^2}{v_a^3}\sqrt\chi_a.
\]

To estimate the second part, we first expand it, obtaining
\begin{align*}
\left\lp\cv^2 v_a,  \cv^2\Big(\frac{|\cv v_a|^2}{v_a}\Big)\right\rp
	&=\frac{2}{v_a}\cv^3v_a(\cv^2 v_a, \cv v_a) + \frac{2}{v_a}\mr{tr}(\cv^2 v_a)^3-\frac{4}{v_a^2}(\cv^2 v_a)^2(\cv v_a,\cv v_a)\\
	&\quad+\frac{2}{v_a^3}\cv^2v_a(\cv v_a, \cv v_a)|\cv v_a|^2-\frac{1}{v_a^2}\chi_a|\cv v_a|^2.
\end{align*}
Applying the weighted Cauchy--Schwarz inequality ($|\alpha\beta|\leq\epsilon \alpha^2 +\frac{1}{4\epsilon} \beta^2$), we initially estimate
\[
2\left| \left\lp\cv^2 v_a,  \cv^2\Big(\frac{|\cv v_a|^2}{v_a}\Big)\right\rp \right|
	\leq |\cv^3 v_a|^2 + 2(2+4+1)\frac{|\cv v_a|^2}{v_a^2}\chi_a +4\frac{\chi_a^{3/2}}{v_a}+4\frac{|\cv v_a|^4}{v_a^3}\sqrt\chi_a.
\]
We then use Young's inequality,
\[
4\frac{\chi_a^{3/2}}{v_a}\leq \ve\chi_a^2+\frac{C}{v_a^4},
\]
and the weighted Cauchy--Schwarz inequality again,
\[
4\frac{|\cv v_a|^4}{v_a^3}\sqrt\chi_a\leq\frac{|\cv v_a|^2}{v_a^2}\chi_a+C\frac{|\cv v_a|^6}{v_a^4},
\]
which, together with the bound $|\cv v_a|^2\leq C$ from Lemma~\ref{C1 estimate}, yields the refinement
\[
2\left| \left\lp\cv^2 v_a,  \cv^2\Big(\frac{|\cv v_a|^2}{v_a}\Big)\right\rp \right|
	\leq |\cv^3 v_a|^2+15\frac{|\cv v_a|^2}{v_a^2}\chi_a+\ve\chi_a^2+\frac{C}{v_a^4}.
\]

If the base is two-dimensional, then by Lemma~\ref{C0 estimates} and Lemma~\ref{lemma-checkR-upper}, we have
\[
4\big|\check{\Rm}(\check\cv^2v_a,\check\cv^2v_a)\big|\leq\frac{C}{v_1^2}\chi_a \leq \ve\sum_b\chi_b^2+C'\sum_b v_b^{-4}.
\]
So, collecting terms and again using $|\cv v_a|^2\leq C$ and the weighted Cauchy--Schwarz inequality, we find that
\begin{align}
(\chi_a)_t	&\leq\Delta\chi_a-|\cv^3v_a|^2+\mu_a\Big(\frac{\chi_a}{2v_a^2}+\frac{|\cv v_a|^2}{v_a^3}\sqrt\chi_a\Big)+15\frac{|\cv v_a|^2}{v_a^2}\chi_a		\notag\\
		&\quad +4\sum_b n_b v_b^{-2}|\cv v_b|\,|\cv v_a|\,\sqrt\chi_b\,\sqrt\chi_a +\ve\chi_a^2+ Cv_a^{-4}
			+ C\sum_b v_b^{-4}|\cv v_b|^4		\notag\\
		&\qquad+\ve\sum_b\chi_b^2+C\sum_b v_b^{-4}	\notag\\
		&\leq\Delta\chi_a-|\cv^3 v_a|^2+\ve\chi_a^2+\ve\sum_b\chi_b^2+C\sum_b v_b^{-4}.		\label{chi_t intermediate}
\end{align}
To obtain the final inequality above, estimate~\eqref{chi_t intermediate}, we use the estimate
\[
4\sum_b n_b v_b^{-2}|\cv v_b|\,|\cv v_a|\,\sqrt\chi_b\,\sqrt\chi_a\leq \ve\sum_b\chi_b\chi_a+C\sum_b v_b^{-4}.
\]
\medskip

Next we recall from Lemma~\ref{C1 estimate} that
\[
(|\cv v_a|^2)_t=\Delta(|\cv v_a|^2)-2\chi_a-4\frac{\cv^2v_a(\cv v_a,\cv v_a)}{v_a}+2\frac{|\cv v_a|^2}{v^2_a}(\mu_a+|\cv v_a|^2).
\]
It follows from this and~\eqref{chi_t intermediate} by straightforward computation that 
\begin{align*}
(|\cv v_a|^2\chi_a)_t	&\leq \Delta(|\cv v_a|^2\chi_a)-2\lp\cv |\cv v_a|^2,\cv \chi_a\rp-|\cv v_a|^2\,|\cv^3 v_a|^2\\
	&\qquad+|\cv v_a|^2\Big(\ve\chi_a^2+\ve\sum_b\chi_b^2+C\sum_b v_b^{-4}\Big)\\
	&\quad-2\chi_a^2 + 4\frac{|\cv v_a|^2\chi_a^{3/2}}{v_a}+2\frac{|\cv v_a|^2\,\chi_a}{v_a^2}(\mu_a+|\cv v_a|^2).
\end{align*}
To improve this estimate, for $\beta>0$ to be chosen below, we use estimate
\[
2\big|\lp\cv |\cv v_a|^2,\cv \chi_a\rp\big|\leq8|\cv^3v_a|\,|\cv^2v_a|^2\,|\cv v_a|\leq\beta|\cv v_a|^2|\cv^3v_a|^2+\frac{16}{\beta}\chi_a^2,
\]
the consequence of Young's inequality, which says that
\[
4\frac{|\cv v_a|^2\chi_a^{3/2}}{v_a}\leq\ve\chi_a^2+\frac{C}{v_a^4},
\]
and the consequence of Lemma~\ref{C1 estimate} and the weighted Cauchy--Schwarz inequality that
\[
2\frac{|\cv v_a|^2\chi_a}{v_a^2}\big(\mu_a+|\cv v_a|^2\big)\leq\ve\chi_a^2+\frac{C}{v_a^4},
\]
to obtain the further refinement
\begin{align}
(|\cv v_a|^2\chi_a)_t	&\leq \Delta(|\cv v_a|^2\chi_a)+(\beta-1)|\cv v_a|^2\,|\cv^3v_a|^2+\Big(\ve+\frac{16}{\beta}-2\Big)\chi_a^2	\label{second step}\\
	&\qquad+|\cv v_a|^2\big(\ve\chi_a^2+\ve\sum_b\chi_b^2+C\sum_b v_v^{-4}\big).	\notag
\end{align}
\medskip

Finally, we combine estimates~\eqref{chi_t intermediate} and~\eqref{second step} to see that
\begin{align*}
\big[(B+|\cv v_a|^2)\chi_a\big]_t	&\leq\Delta\big[(B+|\cv v_a|^2)\chi_a\big]+\big\{(\beta-1)|\cv v_a|^2-B\big\}|\cv^3v_a|^3\\
	&\qquad+\big(\ve+\frac{16}{\beta}-2\big)\chi^2_a+\big(B+|\cv v_a|^2\big)\Big\{\ve\sum_b\chi_b^2+C\sum_b v_v^{-4}\Big\}.	
\end{align*}
We choose $\beta=32$, $B\geq\beta\max\{|\cv v_b|^2\}$, and $\ve\in(0,\frac14)$ small enough that $2B\ve\leq\frac12$.
We note that we can make $\ve$ as small as we wish by increasing $C$ if necessary. Thus it follows from these choices that the
quantity $F=\sum_a (B+|\cv v_a|^2)\chi_a$ defined above satisfies
\[
F_t\leq\Delta F-\sum_a\chi_a^2+C\sum_a v_a^{-4}
	\leq\Delta F+C\sum_a v_a^{-4}.
\]
The sum $\sum_a v_a^{-4}$ is finite as long as each $v_a>0$. 
Because there exists a first time $T>0$ such that any $v_a=0$, the conclusion follows readily.
\end{proof}

By integrating the estimate in Theorem~\ref{Hessian bound}, one immediately obtains:

\begin{corollary}		\label{Hessian Type-I}
If there exists $c>0$ such that $\min_{x\in\mc B}v_a(x,t)\geq c\sqrt{T-t}$ for all $1\leq a\leq A$, then
\[F_{\max}\leq\frac{C}{T-t}.\]
\end{corollary}

Combining this with Lemma~\ref{C1 estimate} and the curvature formulas in Appendix~\ref{Curvature formulas}, one then obtains:

\begin{corollary}		\label{Type-I}
If there exists $c>0$ such that $\min_{x\in\mc B}v_a(x,t)\geq c\sqrt{T-t}$ for all $1\leq a\leq A$ and at least one $v_b\to0$ as $t\nearrow T<\infty$,
then the solution encounters a Type-I singularity at $T$.
\end{corollary}
In Sections~\ref{Generalized cylinders-1} and~\ref{Generalized cylinders-2}, we establish sufficient conditions for the Corollary to hold for
one-dimensional and two-dimensional bases, respectively.
\bigskip

We recall that Lemma~\ref{lemma-initial} applies to any solution satisfying Assumption~\ref{single fiber pinching}.

\begin{theorem}			\label{shrinker-limit}
Suppose a solution flowing from initial data satisfying Assumption~\ref{single fiber pinching} satisfies the Type-I hypotheses of
Corollary~\ref{Type-I} as $t\nearrow T<\infty$. Suppose also that the base is $\mb S^1$ or a surface $\mc B^2$ that also satisfies Assumption~\ref{eta tame}.
And suppose that $\inf_{t\nearrow T}v_a(\cdot,t)=0$ for all $1\leq a\leq A'<A$ but not for any $A'+1\leq a\leq A$.

We define rescaled metrics $\tilde{g}(\cdot,\tau) := (T-t)^{-1}\, g(\cdot,t)$, where $\tau := -\log(T-t)$.
And we define $N_{\mr{gs}}:=n+\sum_{a=1}^{A'} n_a$ and $N_{\mr{fl}}:=\sum_{a=A'+1}^A n_a$.

Then as $\tau_i\to\infty$, the solutions $\big(\mc M, \tilde{g}(\cdot,\tau_i)\big)$ converge subsequentially  smoothly and locally uniformly to
\[\
\big(\mc K^{N_{\mr{gs}}} \times \mb R^{N_{\mr{fl}}}, g_\infty\big),
\]
where $\mc K^{N_{\mr{gs}}}$ is a nonflat gradient shrinking Ricci soliton and $\mb R^{N_{\mr{fl}}}$ is flat.
\end{theorem}

\begin{proof}
By the Type-I hypothesis, the rescaled metrics satisfy $|\Rm[\tilde{g}](\cdot,\tau)| \le C$ for all $\tau\in [-\log T, \infty)$.
By~\cite{EMT11}, this implies that for any sequence $\tau_i \to \infty$ and Type-I singular points $o_i$, there is smooth subsequential convergence
of the pointed solutions $(\mc M, \tilde{g}(\cdot,\tau + \tau_i),o_i)$
to a Ricci flow solution $(\mc M_{\infty}, g_{\infty}(\cdot,\tau), o_{\infty})$ that is a nonflat gradient shrinking Ricci soliton.

We claim that  the limit splits $N_{\mr{fl}}$ flat directions and hence has an $\mb R^{N_{\mr{fl}}}$ factor.
We recall that, without rescaling, the norm of the curvature is given by
\[
|\Rm|^2 = g^{I\tilde{I}} g^{J\tilde{J}} g^{K\tilde{K}} g^{L\tilde{L}} R_{IJKL} R_{\tilde{I}\tilde{J}\tilde{K}\tilde{L}},\]
where we may assume that both vectors in each index pair $(I,\tilde{I})$, $(J,\tilde{J})$, $(K,\tilde{K})$ and $(L,\tilde{L})$ are tangent to the same factor.
We write this norm as
\[
|\Rm|^2 = \Sigma_{{\mr{fl}}} + \Sigma_{{\mr{gs}}},
\]
where $\Sigma_{{\mr{fl}}}$ is the sum of all the terms with at least one index pair corresponding to a fiber $a\in\{A',\dots,A\}$, and
$\Sigma_{{\mr{gs}}}$ is the sum of all the other terms.

In order to see that the limit splits off a Euclidean factor of dimension $N_{\mr{fl}}$, it suffices to show that the rescaled sum
$\tilde\Sigma_{{\mr{fl}}} = (T-t)^2 \Sigma_{{\mr{fl}}}$ tends to zero as $\tau\to \infty$. To see this, we recall the formulas in
Appendix~\ref{Curvature formulas} and consider all the terms that contribute to $\tilde\Sigma_{{\mr{fl}}}$, that is, all terms corresponding to the
fibers $a\in\{A',\dots,A\}$. The four types of curvature considered in Appendix~\ref{Curvature formulas} are as follows:

\begin{enumerate}
\item[(a)]
We do not need to consider $\check R$, because this corresponds to a term in $\Sigma_{{\mr{gs}}}$.

\item[(b)]
Curvature terms that pair the $a^{\mr{th}}$ fiber with itself: by formula~\eqref{fiber-self}, these terms in $\tilde\Sigma_{{\mr{fl}}}^{1/2}=(T-t)\Sigma_{{\mr{fl}}}^{1/2}$
are bounded by
\[
C\,(T-t)\frac{1+|\nabla v_a|^2}{v_a^2},
\]
where $C$ is a uniform constant. Since $v_{\alpha} \ge \delta > 0$ by Lemma~\ref{lemma-initial}, and Lemma~\ref{C1 estimate} applies, these terms
converge to zero as $t\to T$, equivalently as $\tau\to\infty$.

\item[(c)]
Curvature components that pair the $a^{\rm{th}}$ fiber with any different fiber: by formula~\eqref{fiber-other}, these terms in
$\tilde\Sigma_{{\mr{fl}}}^{1/2}=(T-t)\Sigma_{{\mr{fl}}}^{1/2}$  are bounded by
\[
C\, (T-t) \frac{|\nabla v_a||\nabla v_b|}{v_{a}v_{b}}.
\]
By the Type-I assumption that $\min_{x\in\mc B} v_b(x,t) \ge c\sqrt{T-t}$ for some $c > 0$ and all $b\in\{1,\dots,A\}$, the consequence of
Lemma~\ref{lemma-initial} that $v_a \ge \delta > 0$ for all $a\in\{A'+1,A\}$, and Lemma~\ref{C1 estimate},
these terms are bounded by $C^*\sqrt{T-t}$, which tends to zero as $t\to T$.

\item[(d)]
Curvature components that pair the base with the $a^{\mr{th}}$ fiber: by~\eqref{base-fiber}, Corollary~\ref{Hessian Type-I}, and the lower bound
$v_a \ge \delta > 0$ for all $a\in\{A'+1,A\}$ again, these terms in $\tilde\Sigma_{{\mr{fl}}}^{1/2}=(T-t)\Sigma_{{\mr{fl}}}^{1/2}$  are bounded by
\[
C\, (T-t)\frac{|\nabla^2 v_{a}|}{v_{a}} \le C^*\sqrt{T-t} \to 0\quad\mbox{ as }\quad t\to T.
\]
\end{enumerate}

This completes the proof.
\end{proof}

\begin{corollary}		\label{warped-product-limit}
The soliton $\mc K^{N_{\mr{gs}}}$ constructed in Theorem~\ref{shrinker-limit} is a warped product. 
\end{corollary}

\begin{proof}
For each $\tau$, the rescaled manifold $\big(\mc M^N,\tilde g(\tau)\big)$ is a Riemannian submersion.
Following~\cite{ONB66}, let $\pi_x:T_x\mc M^N \to T_b \mc B^n$ denote the
tangent projection with kernel $\mc V_x$ and orthogonal complenent $\mc H_x$. 

We observe that the size  $\big|\mc F_a\big|_{\tilde g(\tau)}$ of each fiber is uniformly bounded from below by a positive quantity,
the horizontal distribution $\mc H$ is integrable (equivalently, O'Neill's tensor $A$, defined in \S2 of~\cite{ONB66}, vanishes), and the metrics are constant
on the vertical distribution $\mc V$ in the precise sense that $\tilde g(\tau)\big|_{\mc V_x}$ depends only on $b=\pi(x)$.

The projection $\pi_x:T_x\mc M^N \to T_b \mc B^n$ is independent of $\tau$. Thus, these properties persist as $\tau_i\to\infty$ in
any smooth subsequential Cheeger--Gromov limit, which must therefore be a warped product globally.
\end{proof}

\begin{corollary}
Suppose the base surface $\mc B^2$ has genus $\geq1$ or is a large 2-sphere satisfying Assumption~\ref{eta tame} for a sufficiently small $\eta$
depending only on the diameter of the smallest fiber.\ Then the soliton $\mc K^{N_{\mr{gs}}}$ constructed in Theorem~\ref{shrinker-limit} is noncompact.
\end{corollary}

\begin{proof}
Let $\check A$ and $\mr d\check A$ denote the area and measure, respectively of $\mc B^2$ with respect to $\check g$. Then, up to diffeomorphism,
it follows from~\eqref{RF base alt} by a standard variational formula that
\[
\partial_t(\mr d\check A) =\big(-\check R+n|\cv w|^2\big)\,\mr d\check A\quad\Rightarrow\quad
\frac{\mr d}{\mr dt}\check A = -4\pi\chi(\mc B^2)+n||\cv w||^2.
\]
Hence, $\check A$ monotonically increases if the genus of $\mc B^2$ is at least 1, and decreases at most linearly if $\mc B^2$ is diffeomorphic to $\mb S^2$.
In the latter case, we may assume that the area is sufficiently large initially, which is consistent with Assumption~\ref{eta tame}.
Thus in either case, after parabolic dilation, the limit base surface has infinite area.
\end{proof}

\section{Multiply warped products over $\mb S^1$}		\label{Generalized cylinders-1}

Throughout this section, we assume a one-dimensional closed base $\mc B^1=\mb S^1$ and we prove Main Theorem~\ref{Main-1}.
Specifically, we consider multiply warped product metrics
on $\mb S^1\times\mb S^{n_1}\times\cdots\times\mb S^{n_A}$ of the form
\[
g = (\mr ds)^2 + \textstyle\sum_{a=1}^A g_a  =  (\mr ds)^2 + \textstyle\sum_{a=1}^A v_a^2\,\hat g_a,
\]
where all $n_a\geq2$ and $s$ denotes the arclength from a fixed but arbitrary point $\theta_0$ on $\mb S^1$.

Using Appendix~\ref{Curvature formulas}, we observe that
\begin{align*}
R_{00}	&= -\left(\sum_a  n_a\frac{(v_a)_{ss}}{v_a}\right)\mr ds^2,\\
(\Rc_a)_{\alpha\beta}	&=
	\left(-\frac{(v_a)_{ss}}{v_a}+(n_a-1)\frac{1-(v_a)_s^2}{v_a^2}-\sum_{b\neq a}n_b\frac{(v_a)_s(v_b)_s}{v_a v_b}\right\}(g_a)_{\alpha\beta}.
\end{align*}
In particular, the scalar curvature of the whole manifold is
\begin{equation}		\label{1-dim scalar}
R=-2\sum_a n_a\frac{(v_a)_{ss}}{v_a}
	+\sum_a n_a(n_a-1)\frac{1-(v_a)_s^2}{v_a^2}
	-\sum_a\sum_{b\neq a} n_a n_b\frac{(v_a)_s(v_b)_s}{v_a v_b}.
\end{equation}

\subsection{The arclength commutator} Using $s$ as a gauge induces a commutator $[\partial_s,\partial_t]$.
Specifically, we may suppose there is locally a smooth function $\phi$ such that
\begin{equation}	\label{density}
s(\theta,t)=\int_{\theta_0}^\theta\phi(\tilde\theta,t)\,\mr d\tilde\theta,
\end{equation}
which implies that
\[
\partial_s=\phi^{-1}\partial_\theta \qquad\mbox{ and }\qquad \mr ds=\phi\,\mr d\theta.
\]
Then $2\phi\phi_t (\mr d\theta)^2= -2 R_{00}\,(\phi\,\mr d\theta)^2$, which is equivalent to
\[
\frac{\phi_t}{\phi}=-R_{00}.
\]
This yields the commutator
\begin{equation}
\big[\partial_t,\,\partial_s\big]
	=\big[\partial_t,\,\phi^{-1}\partial\theta\big]
	=-\frac{\phi_t}{\phi^2}\,\partial_\theta=-\frac{\phi_t}{\phi}\,\partial_s
	=R_{00}\,\partial_s
	= - \sum_a  n_a\frac{(v_a)_{ss}}{v_a}\,\partial_s.	\label{Commutator}
\end{equation}

\subsection{Type-I singularity formation}

We begin by constructing an open set of initial data that can form singularities modeled on generalized cylinders.

\begin{assumpA}		\label{guarantee cylinder}
Each Einstein constant is at least as large as that of the standard sphere, $\mu_a\geq n_a-1$;
 and there are constants $r_0\in\mb R$ and $c_0>0$ sufficiently small such that
the scalar curvature of the whole manifold satisfies $\min R(\cdot,0) \ge r_0$ and the size of the first
fiber satisfies $v_1(\cdot,0)\leq c_1$.
\end{assumpA}

We assume in the remainder of this section that the initial data satisfy both Assumptions~\ref{single fiber pinching} and~\ref{guarantee cylinder}.

\begin{lemma}		\label{lemma-TypeI}
We consider a Ricci flow solution over $\mb S^1$ that evolves from initial data satisfying both Assumptions~\ref{single fiber pinching} and~\ref{guarantee cylinder}.
 Then there exists a uniform constant $c$ so that
\[\min v_a(\cdot,t) \ge c\sqrt{T-t}, \qquad t\in [0,T),\]
for all fibers with $1 \le a \le A$.
\end{lemma}

\begin{proof}
Since $R_t=\Delta R+2|\Rc|^2\geq\Delta R$, we have $R\geq r_0$ for as long as the solution remains smooth.

We define $Q:=\log\big(v_1^{2n_1}\,v_s^{2n_2}\cdots v_A^{2n_A}\big)$. Recalling the evolution equation~\eqref{RF fibers} for $v_a$
and using~\eqref{Laplacian} in the form
$\Delta v_a = (v_a)_{ss}+\sum_b n_b\frac{(v_a)_s(v_b)_s}{v_b}$, we then expand $Q_t=\sum_a2n_a\frac{(v_a)_t}{v_a}$, yielding
\begin{equation}		\label{evolve Q}
Q_t=2\sum_a n_a\Big\{\frac{(v_a)_{ss}}{v_a}+\sum_b n_b\frac{(v_a)_s(v_b)_s}{v_av_b}-\frac{\mu_a+(v_a)_s^2}{v_a^2}\Big\}.
\end{equation}
By~\eqref{1-dim scalar}, \eqref{evolve Q}, and the fact that $R\geq r_0$ is preserved, we obtain the inequality
\begin{equation}		\label{eq-Q1}
Q_t\leq -r_0+\sum_a n_a\frac{n_a-1-2\mu_a}{v_a^2}+\sum_a n_a(n_a-1)\frac{(v_a)_s^2}{v_a^2}+\sum_{b\neq a}n_an_b\frac{(v_a)_s(v_b)_s}{v_av_b}.
\end{equation}

We observe that at any point in space where $Q_{\min}(t)$ is attained, one has
\[
\sum_a n_a\frac{(v_a)_s}{v_a}=0,
\]
which implies that at any such point,
\[
\sum_{b\neq a}n_an_b\frac{(v_a)_s(v_b)_s}{v_av_b}=\sum_a \left(n_a\frac{(v_a)_s}{v_a}\Big(\sum_{b\neq a}n_b\frac{(n_b)_s}{n_b}\Big)\right)
=-\sum_a n_a^2 \frac{(v_a)_s^2}{v_a^2}.
\]
Combining this with~\eqref{eq-Q1} yields
\[
\frac{\mr d}{\mr dt}Q_{\min}\leq-r_0+\sum_a n_a\frac{n_a-1-2\mu_a}{v_a^2}-\sum_a n_a\frac{(v_a)_s^2}{v_a^2},
\]
which by our assumption on $\mu_a$ implies that
\begin{equation}	\label{eq-Q2}
\frac{\mr d}{\mr dt}Q_{\min}\leq-r_0-\sum_a \frac{n_a\mu_a}{v_a^2}.
\end{equation}

We define $P:=v_1^{n_1}\,v_2^{n_2}\cdots v_A^{n_A}$, so that $Q=2\log(P)$. Then estimate~\eqref{eq-Q2} implies that at any spatial minimum of $P$, one has
\[
P_t	=\frac12 PQ_t	\leq\frac12 P\Big(-r_0-\sum_a \frac{n_a\mu_a}{v_a^2}\Big)
	\leq\frac12 v_1^{n_1-2}v_2^{n_2}\cdots v_a^{n_A}\big(-r_0 v_1^2-n_1\mu_1\big).
\]
A consequence of Lemma~\ref{C0 estimates} is $v_1(\cdot,t) \le \max v_1(\cdot, 0) =c_1$.
Thus for $c_1$ sufficiently small, depending on $r_0$ and $n_1\mu_1$, there exists $c_2>0$ such that $\frac12(r_0 v_1^2+n_1\mu_1\big)\geq c_2$.
Furthermore, by our assumed lower bounds for $v_2,\dots,v_A$ and the upper bounds given by Lemma~\ref{C0 estimates}, the product $P$ is comparable to $v_1^{n_1}$.
Thus there exists $c_3>0$ such that
\[
\frac{\mr d}{\mr dt}P_{\min}\leq-c_2\,v_1^{n_1-2}v_2^{n_2}\cdots v_a^{n_A}\leq -c_3\,P^{1-\frac{2}{n_1}}.
\]
Integrating this from $t$ to $T$ and using $P_{\min}(T)=0$ and the fact that $P^{\frac{1}{n_1}}$ and $v_1$ are comparable, we obtain $c_4,c_5>0$ such that
\[
P(t)^{\frac{2}{n_1}}\geq c_4(T-t)\qquad\Rightarrow\qquad v_1^2(\cdot,t)\geq c_5(T-t).
\]
This completes the proof.
\end{proof}

Thus we have the following estimate for $v_{\min}:=\min\{v_a:1\le a\le A\}$:

\begin{corollary}		\label{cor-vmin}
Assume a solution over $\mb S^1$ originates from initial data satisfying both Assumptions~\ref{single fiber pinching} and~\ref{guarantee cylinder}.
Then there exist uniform constants $0<c<C<\infty$ such that $c\,(T-t) \le v_{\min}^2(t) \le (T - t)$ for all $t\in [0,T)$.
\end{corollary}

\begin{proof}
Consequences of Lemma~\ref{lemma-TypeI} are that $\liminf_{t\to T} \big(v_{\min}(t)\big) = 0$ and $v_{\min}(t) \ge c\,(T-t)$ for all $t\in [0,T)$, where $c$ is a uniform positive constant.
The upper bound follows immediately from Lemma~\ref{C0 estimates}.
\end{proof}

\begin{remark}		\label{Almost cylinder}
Lemma~\ref{lemma-TypeI} and Corollary~\ref{cor-vmin} now allow us to invoke Corollary~\ref{Type-I}, Theorem~\ref{shrinker-limit}, and
Corollary~\ref{warped-product-limit} from Section~\ref{Elementary estimates}, yielding a limit
\[
\big(\mc K^{N_{\mr{gs}}} \times \mb R^{N_{\mr{fl}}}, g_\infty\big),
\]
where $\mc K^{N_{\mr{gs}}}$ is a nonflat gradient shrinking soliton of dimension $N_{\mr{gs}}=1+n_1$ with a warped product structure.
\end{remark}
We next show that $\mc K^{N_{\mr{gs}}}$ is a cylinder by analyzing $(v_1)_{ss}$.
\medskip

We recall that $v_1$ controls the size of the fiber that crushes. To get better control on its second derivative, we now consider
\[
L := v_1 (v_1)_{ss}\, \log v_1.
\]
Motivation to consider this quantity comes from a paper of the second author~\cite{AK04}.
In the following theorem, we show that $L$ is bounded from below in a suitable spacetime neighborhood of a local singularity.
\medskip

We define that neighborhood as follows: for fixed $0 < \delta \ll 1$, Corollary~\ref{cor-vmin} implies that there exists $t_{\delta}\in [0,T)$
such that the radius of each neck (each local minimum of $v_1$) that becomes singular satisfies $v_1 \le \delta$ for all $t_{\delta} \le t < T$.
By~\eqref{RF fibers}, the Sturmian theorem~\cite{Sturmian}  applies to each $v_a$ and implies that critical points are nondegenerate, except
possibly where two critical points merge. Therefore, we may assume that $(v_1)_{ss}>0$ at each local minimum of $v_1$. It follows that
\[
\Omega := \left\{s\in\mb S^1\colon (v_1)_{ss} \, \log\left(\frac{v_1}{\delta}\right) < 0\right\}
\]
is the union of an open interval around each neck for all $t\in (t_\delta,T)$.

\begin{assumpA}		\label{small gradient}
$\mu_1 = n_1 -1$ and $(v_1)_s^2(\cdot,0) \le 1$.
\end{assumpA}

By Lemma~\ref{C1 estimate}, it follows from this assumption that $(v_1)_s^2(s,t) \le 1$ for as long as a smooth solution exists.

\begin{theorem}		\label{Neck}
Under Assumptions~\ref{single fiber pinching},~\ref{guarantee cylinder}, and~\ref{small gradient}, there exists a uniform positive
constant $C$ such that for as long as the flow exists,
\[
L \ge -C\quad\mbox{ in }\quad\Omega.
\]
\end{theorem}

\begin{proof}
Corollary~\ref{Hessian Type-I} and Corollary~\ref{cor-vmin} imply $L \ge C\, \log(T-t_{\delta})/\sqrt{T - t_{\delta}}$, at $t = t_{\delta}$.
The definition of $\Omega$ guarantees $L = 0$ at the boundary of each component of $\Omega$ for all times $t_{\delta} < t < T$.
To complete the proof, we need to show that $L$ is bounded from below at all interior points for all $t\in [t_{\delta},T)$.
We show this by applying the maximum principle to the evolution of $L$, using the facts that $(v_1)_{ss} > 0$ and $v_1 < \delta$ inside $\Omega$.

A straightforward but tedious computation yields
\[
\frac{\partial}{\partial t} L = \Delta L -\frac{2\, L_s(v_1)_s}{v_1}\Big(2 + (\log v_1)^{-1}\Big) + \mc N,
\]
where
\begin{align*}
\mc N &:= -\frac{2\mu_1 (\log v_1) (v_1)_s^2}{v_1^2} - \frac{2(\log v_1)(v_1)_s^4}{v_1^2} + 2v_1(\log v_1) \sum_b \frac{n_b (v_1)_s (v_b)_s^3}{v_b^3} \\
&\qquad- \frac{\mu_1 \, (v_1)_{ss}}{v_1} + \frac{4(v_1)_s^2(v_1)_{ss}}{v_1} + \frac{2(v_1)_s^2 (v_1)_{ss}}{v_1\, \log v_1}
	+ (8 - 4n_1)\, \frac{(\log v_1) (v_1)_s^2 (v_1)_{ss}}{v_1}\\
&\qquad-2\, v_1(\log v_1)(v_1)_{ss}\sum_{b>1} \frac{(v_b)_s^2}{v_b^2} - 2(\log v_1)(v_1)_{ss}^2 - 2v_1(\log v_1)\sum_{b>1} \frac{(v_1)_s(v_b)_s(v_b)_{ss}}{v_b^2}.
\end{align*}

We now derive several inequalities. Using that $\log v_1 < 0$ and $(v_1)_{ss} > 0$ in $\Omega$, we estimate that
\begin{multline*}
-\frac{2\mu_1\, \log v_1 (v_1)_s^2}{v_1^2} - \frac{2\log v_1 (v_1)_s^4}{v_1^2} + \frac{2\,n_1\, \log v_1\, (v_1)_s^4}{v_1^2} \\
= -\frac{2\, \log v_1}{v_1^2} (v_1)_s^2\Big(\mu_1 - (n_1 - 1)\, (v_1)_s^2\Big) \ge 0.
\end{multline*}
Furthermore, we may estimate that
\[2\, v_1(\log v_1)\sum_{b>1}  \frac{n_b (v_1)_s (v_b)_s^3}{v_b^3} \ge -C\]
for a uniform constant $C>0$, since $v_b \ge \eta > 0$ for all $b \in \{2,\dots A\}$.
We also have
\[\frac{4(v_1)_s^2(v_1)_{ss}}{v_1} + \frac{2(v_1)_s^2 v_{ss}}{v_1\,\log v_1} + (8 - 4n_1)\, \frac{\log v_1 (v_1)_s^2 (v_1)_{ss}}{v_1} \ge 0\]
if $\delta > 0$ is small enough. Moreover, since $n_1 \ge 2$, we have
\[-2v_1(\log v_1)\sum_{b>1}  \frac{(v_1)_s(v_b)_s(v_b)_{ss}}{v_b^2} \ge -\frac{C}{\sqrt{T-t}},\]
where we have used Lemma~\ref{lemma-initial}, Corollary~\ref{Hessian Type-I}, and Lemma~\ref{lemma-TypeI}.

Combining the inequalities derived above yields
\[\frac{d}{dt}L_{\min} \ge -C (T-t)^{-1/2} - 2\, \frac{(v_1)_{ss}}{v_1}\, \Big(L_{\min} + \frac{\mu_1}{2}\Big).\]
This implies at any interior point that either $L_{\min}(t) \ge -\frac{\mu_1}{2}$ or else that
\[\frac{d}{dt}\, L_{\min} \ge -C (T-t)^{-1/2},\]
which also implies a lower bound on $L_{\min}(t)$, since the right-hand side is integrable.
\smallskip

Because $L=0$ at each point of $\partial\Omega$ and $L_{\min}(t_{\delta}) \ge C\, \log(T-t_{\delta})/\sqrt{T-t_{\delta}}$, this complete the proof.
\end{proof}

With Theorem~\ref{Neck} in hand, we are now ready to complete the following:

\begin{proof}[Proof of Main Theorem~\ref{Main-1}]
By Remark~\ref{Almost cylinder}, we have a limit $\big(\mc K^{N_{\mr{gs}}} \times \mb R^{N_{\mr{fl}}}, g_\infty\big)$,
where $\mc K^{N_{\mr{gs}}}$ is a nonflat gradient shrinking soliton of dimension $N_{\mr{gs}}=1+n_1$. 
By the proof of Theorem~\ref{shrinker-limit}, the sectional curvatures of the limit $\mc K^{N_{\mr{gs}}}$ are convex linear combinations of the limits of
\[
\kappa_0:=-\frac{(v_1)_{ss}}{v_1}\qquad\mbox{ and }\qquad\kappa_1:=\frac{1-(v_1)_s^2}{v_1^2}.
\]

We recall the density $\phi$ of $s$ defined in~\eqref{density}. 
The commutator~\eqref{Commutator} shows that $\phi$ of $s$ is increasing in $\Omega$, hence that the diameter
of $\Omega$ is increasing. It follows that the nonflat gradient shrinking soliton $\mc K^{N_{\mr{gs}}}$ is noncompact, hence that it has the topology of
$\mb R\times\mb S^{n_1}$. Estimate~\eqref{F-estimate} implies that on that soliton, $(\kappa_0)_\infty=0$ and $(\kappa_1)_\infty$ is constant and nonzero in space.
Therefore, the limit must be the shrinking cylinder soliton; \emph{i.e.,} the Gaussian metric on $\mb R\times\mb S^{n_1}$.

More precisely, Theorem~\ref{Neck} implies the bound
\begin{equation}		\label{F-estimate}
|\kappa_0|\leq\frac{C}{v_1^2\,|\log v_1|}.
\end{equation}
Because $v_1\searrow0$ as $t\nearrow T$, this implies upon parabolic rescaling that $|\tilde\kappa_0|\to 0$ in the rescaled region $\tilde\Omega$.
With this estimate in hand, the asymptotics are proved exactly as in Section~9 of~\cite{AK04}. We omit further details.
This completes the proof.
\end{proof}

\section{Multiply warped products over closed surfaces}		\label{Generalized cylinders-2}

In this section, we consider multiply warped products over a base manifold $\mc B^2$ that is a general closed surface, and we 
prove Theorem \ref{Main-2} and Corollary \ref{cor-n1-2}, using estimates we proved in Section \ref{Elementary estimates}.
Recall that, in our choice of gauge, the Ricci flow system~\eqref{RF gauged system} over $\mc B^2$ becomes
\begin{align*}
\partial_t\check g	&= -\check R\check g+2\sum_a n_a \cv w_a \ten \cv w_a,\\
\partial_t w_a		&=\check\Delta w_a-\mu_ae^{-2 w_a}. 
\end{align*}

\begin{lemma}		\label{lemma-2Dbase-TypeI}
Consider a Ricci flow solution over $\mc B^2$ that evolves from initial data satisfying
Assumptions~\ref{single fiber pinching}, \ref{eta tame}, and~\ref{guarantee cylinder}.
Then there exist a singularity time $T<\infty$ and a uniform constant $c$ so that
\[\min v_a(\cdot,t) \ge c\sqrt{T-t}\]
for all fibers $1 \le a \le A$ and all times $0\leq t< T$.
\end{lemma}

\begin{proof}

We define $\mc Q:=2\sum_a n_a w_a$ and note that it follows from~\eqref{RF fibers} that
\[
	(\partial_t-\check\Delta)\mc Q = 2\sum_{a}n_a\Big(-\mu_a v_a^{-2}+ \sum_b n_b\lp\cv w_a,\cv w_b \rp \Big).
\]
We recall from Appendix~\ref{Curvature formulas} that the Ricci curvatures are
\[
	R_{ij} = \tfrac{1}{2}\check R \check g_{ij} - \sum_a n_a v_a^{-1}(\check \cv^2 v_a)_{ij},
\]
and for each fiber $\mc F_a$ with $a\in\{1,\dots,A\}$,
\[
	R_{\alpha\alpha} = \Big(- n_a v_a^{-1}\check\Delta v_a 
	+ \mu_a v_a^{-2} - (n_a-1) |\cv w_a|^2
	- \sum_{b\neq a} n_b \lp \cv w_a, \cv w_b \rp\Big)(g_a)_{\alpha \alpha}.
\]
Thus we see that
\[
	R = \check R - 2\sum_{a} n_a v_a^{-1}\check \Delta v_a
	+ \sum_a n_a\mu_a v_a^{-2}
	+ \sum_a n_a |\cv w_a|^2
	- \sum_{a, b} n_a n_b \lp \cv w_a, \cv w_b \rp.
\]

Since $\check\Delta w_a = v_a^{-1}\check\Delta v_a - |\cv w_a|^2$, we have
\begin{align*}
	\mc Q_t &= 2\sum_a n_a v_a^{-1}\check\Delta v_a - n_a|\cv w_a|^2
	- n_a\mu_a v_a^{-2} + \sum_b n_an_b\lp \cv w_a, \cv w_b \rp\\ 
	&= \check R - R - \sum_a n_a \Big(|\cv w_a|^2 + \mu_a v_a^{-2}- \tfrac{1}{2}\lp \cv w_a, \cv\mc Q\rp \Big).
\end{align*}
At any point in space where $\mc Q_{\min}(t)$ is attained, we have $\cv \mc Q=0$ and thus
\[
	\frac{\mr d}{\mr dt}\mc Q_{\min}
	\leq \check R - R - n_1\mu_1 v_1^{-2}.
\]

Because $R_t=\Delta R+2|\Rc|^2\geq\Delta R$, we have $R\geq r_0\in\mb R$ for as long as the solution remains smooth.
It follows from Lemma~\ref{lemma-checkR-upper} that $\check R\le \max\big\{C_0, \frac{2\mu_1}{3v_1^2}\big\}$).

Let $\mc P=e^{\mc Q/2}=v_1^{n_1}\cdots v_A^{n_A}.$
Then we have
\[
\frac{\mr d}{\mr dt}\mc P_{\min}
	\leq  \tfrac{1}{2} v_1^{n_1-2}\cdots v_A^{n_A}
	\left(- n_1\mu_1  + \tfrac{2}{3}\mu_1 - r_0v_1^2  \right).
\]
It follows from Lemma~\ref{C0 estimates} $v_1(\cdot,t) \le \max v_1(\cdot, 0) =c_1$.
Thus for $c_1$ sufficiently small, depending on $r_0$ and $n_1\mu_1$, there exists $c_2>0$ such that $\frac12(r_0 v_1^2+n_1\mu_1-\frac{2}{3}\mu_1\big)\geq c_2$.
Furthermore, by the upper bounds given by Lemma~\ref{C0 estimates} and the lower bounds from Lemma~\ref{lemma-initial}, 
the product $\mc P$ is comparable to $v_1^{n_1}$. Thus there exists $c_3>0$ such that
\[
\frac{\mr d}{\mr dt}\mc P_{\min}\leq-c_2\,v_1^{n_1-2}v_2^{n_2}\cdots v_a^{n_A}\leq -c_3\,\mc P^{1-\frac{2}{n_1}}.
\]
Integrating this from $t$ to $T$ and using $\mc P_{\min}(T)=0$ and the fact that $\mc P^{\frac{1}{n_1}}$ and $v_1$ are comparable, there exist $c_4,c_5>0$ such that
\[
\mc P(t)^{\frac{2}{n_1}}\geq c_4(T-t)\qquad\Rightarrow\qquad v_1^2(\cdot,t)\geq c_5(T-t).
\]
This completes the proof.
\end{proof}

With this result in hand, one obtains an exact analog of Corollary~\ref{cor-vmin} by replacing Lemma~\ref{lemma-TypeI} with
Lemma~\ref{lemma-2Dbase-TypeI}, yielding the following conclusion:

\begin{corollary}		\label{cor-2D-vmin}
Assume a solution over $\mc B^2$ originates from initial data that satisfy Assumptions~\ref{single fiber pinching}, \ref{eta tame}, and~\ref{guarantee cylinder}.
Then there exist uniform constants $0<c<C<\infty$ such that $c\,(T-t) \le v_{\min}^2(t) \le (T - t)$ for all $t\in [0,T)$.
\end{corollary}

We are now ready to prove Theorem \ref{Main-2}.

\begin{proof}[Proof of Theorem \ref{Main-2}]
Recalling the two-sided bounds for $\check R$ given by Lemmas~\ref{lemma-checkR-upper} and~\ref{lemma-checkR-lower}, using
Lemma~\ref{lemma-2Dbase-TypeI} and Corollary~\ref{cor-2D-vmin}  now allows us to invoke Corollary~\ref{Type-I}, Theorem~\ref{shrinker-limit}, and
Corollary~\ref{warped-product-limit} from Section~\ref{Elementary estimates}, yielding a direct product limit
\[
\big(\mc K^{N_{\mr{gs}}} \times \mb R^{N_{\mr{fl}}}, g_\infty\big),
\]
where $g_{\infty} = g_{\mc K_{N_{\mr{gs}}}} + g_{fl}$, and $(\mc K^{N_{\mr{gs}}}, g_{K_{\mr{gs}}}) $ is a nonflat gradient shrinking soliton of dimension $N_{\mr{gs}}=2+n_1$ with a warped product structure over a complete noncompact two dimensional base $\tilde{\mc B}$ on $K_{\mr{gs}} = \tilde{\mc{B}}\times S^{n_1}$. Here $N_{fl} = n_2 + \dots n_A$.
\end{proof}

We now prove Corollary \ref{cor-n1-2}.

\begin{proof}[Proof of Corollary \ref{cor-n1-2}]
Let us assume that $n_1 = 2$. Without loss of generality, the Assumptions can be arranged so that the fiber that crushes is $\mathbb{S}^2$. By Theorem \ref{Main-2}, we know that  our singularity is Type-I, and after Type-I rescaling, our singularity model is  $\big(\mc K^{4} \times \mb R^{N_{\mr{fl}}}, g_\infty\big)$, with a direct product metric $g_{\mc{K}^4} + g_{\rm{eucl}}$, where $N_{fl} = n_2 + \dots n_A$. Here $g_{\rm{eucl}}$ is the Euclidean metric on the factor $\mb R^{N_{\mr{fl}}}$, $\mc{K}^4 = \tilde{\mc B} \times \mb{S}^2$, and $g_{\mc{K}^4} $ is a warped product metric over a two-dimensional base $\tilde{\mc{B}}$. By Theorem \ref{thm-classification}, we know that the shrinker $(\mc{K}^4, g_{\mc K^4})$ is isometric to a generalized cylinder $\mb R^2\times\mb S^2$ with a standard cylindrical metric, as claimed. 
 \end{proof}

\appendix

\section{Geometric basics}		\label{Geometric basics}

\subsection{Metrics}

Here we review the geometry of a mutiply warped product manifold
\[
\Big(\mc M^N:=\mc B^n\times\mc F_1^{n_1}\times\cdots\times\mc F_A^{n_A},\;g:=g_{\mc B} + \sum_{a=1}^A v_a^2\,g_{\mc F_a}\Big),
\]
where $N:=n+\sum n_a$.

If we work in local coordinates, we use lowercase Roman indices $i,j,k\in\{1,\dots,n\}$ on the base and Greek indices
$\alpha,\beta,\gamma\in\{1,\dots,n_a\}$ in each fiber $a\in\{1,\dots,A\}$. We use uppercase Roman indices to range over all possible values.

\subsection{Connections}
Using the convention explained above, one finds that the Levi--Civita connection of~\eqref{MWP metric} has a block structure determined by
\begin{subequations}		\label{Connection}
\begin{align}
\Gamma_{ij}^k&=\check\Gamma_{ij}^k,\\
(\Gamma_a)_{\alpha\beta}^k&=-\frac{\partial_k v_a}{v_a}(g_a)_{\alpha\beta},\qquad
(\Gamma_a)_{i\beta}^\gamma=\frac{\partial_iv_a}{v_a}\,\delta_\beta^\gamma,\\
(\Gamma_a)_{\alpha\beta}^\gamma&=(\hat\Gamma_a)_{\alpha\beta}^\gamma.
\end{align}
\end{subequations}

\subsection{A few elementary formulas}

Let $\vp:\mc B\rightarrow\mb R$ be any smooth function. By recalling some facts from~\cite{CIKS22}, we note the following basic identities:
\begin{align}
|\cv\vp|^2  &= |\check\cv\vp|_{\check g}^2\,,\\
\nabla^2 \vp &= (\check \nabla^2) \vp
	+ \sum_a	\frac{\lp\cv v_a,\cv\vp\rp}{v_a}g_a,						\label{Hessian}\\
\Delta \vp	&= \check\Delta\vp+\sum_a n_a \frac{\lp\cv v_a,\cv\vp\rp}{v_a},		\label{Laplacian}\\
|\cv^2\vp|^2	&= |\check\cv^2\vp|^2_{\check g} + \sum_a n_a \frac{\lp\cv v_a\cv\vp\rp^2}{v_a^2}.	\label{Tensor norm}
\end{align}

\subsection{Curvatures}		\label{Curvature formulas}
Let $\mc A=\{1,\dots,A\}$. Here, we unpack Lemma~33 of~\cite{CIKS22}, where capital Roman indices range over all possible values and we employ
the conventions that $R_{IJKL}=R_{IJK}^M\,g_{LM}$ and hence that
\[
(g_\alpha \KN g_\beta)_{\sigma\tau\nu\omega} = (g_\alpha)_{\sigma\omega}(g_\beta)_{\tau\nu} + (g_\alpha)_{\tau\nu}(g_\beta)_{\sigma\omega}
	- (g_\alpha)_{\sigma\nu}(g_\beta)_{\tau\omega} - (g_\alpha)_{\tau\omega}(g_\beta)_{\sigma\nu}.
\]

In this way, we find that the curvature operator has four flavors of components encoding its sectional curvatures.
To be precise, $R_{ABCD}=\Rm(e_A \wedge e_B,\,e_D \wedge e_C)$ depends on the planes $e_A \wedge e_B$ and $\,e_C\wedge e_D$ as follows:
\begin{subequations}
\begin{align}
	&\mbox{(a) Base paired with base:}	\notag\\
R_{ijk\ell}		&= \check R_{ijk\ell},	\label{base-base} \\ \notag\\
	&\mbox{ (b) Fiber paired with itself:}		\notag\\
R_{\alpha\beta\gamma\delta}	&\in\Big\{a\in\mc A\colon v_a^2(\hat R_a)_{\alpha\beta\gamma\delta}
	-v_a^{-2}|\cv v_a|^2\big((g_a)_{\alpha\delta}(g_a)_{\beta\gamma}-(g_a)_{\alpha\gamma}(g_a)_{\beta\delta}\big)\Big\}, \label{fiber-self}\\ \notag\\
	&\mbox{ (c) Fiber paired with a distinct fiber:}	\notag\\
R_{\alpha\beta\gamma\delta}	&\in\Big\{a\neq b\in\mc A\colon
	-v_a^{-1}v_b^{-1}\lp\cv v_a,\cv v_b\rp\,(g_a)_{\alpha\delta}(g_b)_{\beta\gamma}\Big\}, \label{fiber-other}\\ \notag\\
	&\mbox{(d) Base paired with fiber:}		\notag\\
R_{i\beta\gamma\ell}	&\in\Big\{a\in\mc A\colon  -v_a^{-1}(\check\cv^2v_a)_{i\ell}\,(g_a)_{\beta\gamma}\Big\}. \label{base-fiber}
\end{align}
\end{subequations}

These formulas become more tractable if each fiber is a space form.

\subsubsection*{Caution} We warn the reader that, while $R_{ijk\ell}=\check R_{ijk\ell}$, the situation with Ricci curvature is not so simple. Indeed, one sees easily that
\[
R_{ij}=g^{AB}R_{iABj}=\check g^{k\ell}R_{ik\ell j}+\sum_a (g_a)^{\alpha\beta}R_{i\alpha\beta j}=\check R_{ij}-\sum_a n_a v_a^{-1}(\check\cv^2 v_a)_{ij}.
\]

\section{Hessian evolution}		\label{Hessian evolution}
Here, we compute the evolution of the covariant Hessian of a general warping function $v_a$.
To begin our derivation, we recall the evolution equation~\eqref{RF fibers} for $v_a$.  Denoting the Lichnerowicz Laplacian by $\Delta_\ell$,
we then have the commutator
\begin{equation}		\label{Forcing term}
(\partial_t-\Delta_\ell)\cv^2v_a=\cv^2(\partial_t-\Delta)v_a=-\cv^2 Z_a,
\end{equation}
where
\begin{equation}		\label{Define Z}
Z_a:= \frac{\mu_a+|\cv v_a|^2}{v_a}.
\end{equation}

We next convert the Lichnerowicz Laplacian into the rough Laplacian using
\[
(\Delta_\ell\cv^2 v_a)_{IJ} = (\Delta\cv^2 v_a)_{IJ} +2R_{IPQJ}(\cv^2 v_a)^{PQ}-2(\Rc*\cv^2 v_a)_{IJ}.
\]
We do not need to expand the Ricci terms because they cancel terms that appear in the evolution of $g^{-1}$
where we compute $(|\cv^2 v_a|)_t$.

We schematically decompose
\[
R_{IPQJ}(\cv^2 v_a)^{PQ} = A_{ij}+B_{\sigma\tau}+C_{\sigma\tau}+D_{\sigma\tau}
\]
\begin{align*}
A_{ij}	 		&:= R_{iPQj}(\cv^2v_a)^{PQ},\\
B_{\sigma\tau}	&:= R_{\sigma k\ell\tau}(\cv^2 v_a)^{k\ell},\\
C_{\sigma\tau}	&:=\sum_b v_b^2(\hat R_b)_{\sigma\nu\omega\tau}(\cv^2 v_a)^{\nu\omega},\\
D_{\sigma\tau}
	&:= \Big(-\frac12\sum_b\sum_c v_b^{-1}v_c^{-1} \lp\cv v_b,\cv v_c\rp (g_b \KN g_c)_{\sigma\nu\omega\tau}\Big)
		\Big(\sum_d v_d^{-1}\lp\cv v_d, \cv v_a\rp(g_d)^{\nu\omega}\Big).
\end{align*}
In the derivation, it is useful to recall the curvature formulas from Appendix~\ref{Curvature formulas}.

By using formula~\eqref{Hessian} and the curvatures where at least one plane is tangent to the base, we compute
\begin{align}
A_{ij}		
		&= R_{ik\ell j}(\cv^2v_a)^{k\ell}+\sum_b (R_b)_{i\tau\nu j}(\cv^2 v_a)^{\tau\nu},	\notag\\
		&= \check R_{ik\ell j}(\check\cv^2v_a)^{k\ell}
		-\Big(\sum_b v_b^{-1}(\check\cv^2 v_b)_{ij}(g_b)_{\tau\nu}\Big)
		\Big(\sum_c v_c^{-1}\lp\cv v_c,\cv v_a\rp(g_c)^{\tau\nu}\Big)	\notag\\
		&=\check R_{ik\ell j}(\check\cv^2v_a)^{k\ell}
		-\sum_b n_b v_b^{-2}\lp\cv v_b,\cv v_a\rp\,(\check\cv^2 v_b)_{ij}.	\label{A}
\end{align}

Similarly, by using the same formulas, we get
\begin{align}
B_{\sigma\tau}	
	&= -\sum_b v_b^{-1}(\check\cv^2 v_b)_{k\ell}(\check\cv^2 v_a)^{k\ell}\,(g_b)_{\sigma\tau}	\notag\\
	&=-\sum_b v_b^{-1}\lp\check \cv^2 v_b,\check\cv^2 v_a\rp_{\check g}\,(g_b)_{\sigma\tau}.		\label{B}
\end{align}

Next, by using the fact that each fiber is Einstein with $\hat{\Rc}_b=\mu_b\,\hat g_b$ and again using formula~\eqref{Hessian} , we obtain
\begin{align}
C_{\sigma\tau}	
	&=\sum_b (\hat R_b)_{\sigma\tau}v_b^{-1}\lp\cv v_b,\cv v_a\rp	\notag\\
	&=\sum_b \mu_b v_b^{-3}\lp\cv v_b,\cv v_a\rp(g_b)_{\sigma\tau}.	\label{C}
\end{align}

Finally, we again have recourse to Lemma~33 of~\cite{CIKS22}, using the observation
\begin{equation}	\label{KN mess}
\frac12(g_d)^{\nu\omega}\sum_b\sum_c(g_b \KN g_c)_{\sigma\nu\omega\tau}
	=(n_d-1)(g_d)_{\sigma\tau} + n_d\sum_{c\neq d}(g_c)_{\sigma\tau}
\end{equation}
to calculate that
\begin{subequations}			\label{D}
\begin{align}
D_{\sigma\tau}
	&=-\sum_b(n_b-1)v_b^{-3}|\cv v_b|^2\lp\cv v_b,\cv v_a\rp (g_b)_{\sigma\tau}\\
	&\quad -\sum_b n_b v_b^{-2} \sum_{c\neq b}v_c^{-1}\lp\cv v_b,\cv v_c\rp \lp\cv v_b,\cv v_a\rp (g_c)_{\sigma\tau}.	
\end{align}
\end{subequations}

After accounting for the evolution of the $g^{-1}$ factors in $|\cv^2v_a|^2$, our work above shows that
\begin{align*}
(|\cv^2 v_a|^2)_t	&=2\Big\lp\cv^2 v_a,\; \Delta\cv^2 v_a+2(A+B+C+D)-\cv^2 Z_a\Big\rp\\
		&=\Delta(|\cv^2 v_a|^2)-2|\cv^3 v_a|^2 - 2\lp\cv^2v_a,\cv^2Z_a\rp + 4\big\lp\cv^2 v_a,\;A+B+C+D\big\rp,
\end{align*}
where $Z_a$ is defined in~\eqref{Define Z}.

We expand the final quantity above using~\eqref{A}--\eqref{D} and collect terms to see that
\begin{subequations}			\label{general chi evolution}
\begin{align}
(|\cv^2 v_a|^2)_t	 &= \Delta(|\cv^2 v_a|^2) - 2|\cv^3 u|^2 -2\lp\cv^2v_a,\cv^2Z_a\rp\\
	&\quad +4\check{\Rm}(\check\cv^2v_a,\check\cv^2v_a) -4\sum_b n_b v_b^{-2}\lp\cv v_b,\cv v_a\rp\lp\check\cv^2 v_b,\check\cv^2 v_a\rp_{\check g}\\
	&\quad +4\sum_b n_b\big\{\mu_b-(n_b-1)|\cv v_b|^2\big\} v_b^{-4}\lp\cv v_b,\cv v_a\rp^2\\
	&\quad -4\sum_b \sum_{c \neq b}
		n_b  n_c\,v_b^{-2} v_c^{-2}\lp\cv v_b,\cv v_c\rp\lp\cv v_b,\cv v_a\rp\lp\cv v_c,\cv v_a\rp.
\end{align}
\end{subequations}
This result is summarized as Lemma~\ref{alt chi evolution} above.

\end{document}